\documentclass[11pt,a4paper]{amsart}
\usepackage[hmargin=3cm, vmargin={4cm, 4cm}]{geometry}
\usepackage{amssymb,amsthm,mathtools,enumitem,csquotes,xcolor}
\usepackage[initials,msc-links]{amsrefs} 
\usepackage[english]{babel}
\usepackage[final,tracking=true,kerning=true,spacing=true,
factor=1000,stretch=20,shrink=20,babel=true]{microtype}
\microtypecontext{spacing=nonfrench}
\usepackage{hyperref}

\mathtoolsset{showonlyrefs}
\definecolor{refblue}{HTML}{1002f5}
\definecolor{refgreen}{HTML}{019100}
\hypersetup{colorlinks, 
linkcolor={refgreen}, 
citecolor={refgreen}, 
urlcolor={black}}
\setlist{topsep=6pt, itemsep=0.4em, after=\vspace{2pt}}
\numberwithin{equation}{section}
\newtheorem{mainthm}{Theorem}

\newtheorem{thm}{Theorem}[section]
\newtheorem{cor}[thm]{Corollary}
\newtheorem{lem}[thm]{Lemma}

\theoremstyle{definition}

\theoremstyle{remark}
\newtheorem{rem}[thm]{Remark}

\newcommand{\C}{\mathbb{C}}
\newcommand{\D}{\mathbb{D}}
\newcommand{\R}{\mathbb{R}}
\DeclareMathOperator{\re}{Re}
\DeclareMathOperator{\im}{Im}
\newcommand{\bigO}{\mathcal{O}}

\begin{document}                   


\title[A direct approach to edge universality]
{A direct approach to soft and hard edge universality for random normal matrices}

\author{Joakim Cronvall \and Aron Wennman} 

\begin{abstract}
We develop a unified approach to universality of local scaling limits 
for eigenvalues of random normal matrices, or equivalently for planar 
Coulomb gases at inverse temperature \(\beta=2\). The approach is {\em direct} 
in that it does not rely on expressing the kernels in terms of orthogonal 
polynomials. There are three main results. The first is a proof of universality 
at hard edges with no symmetry assumptions on either the potential or the hard edge. 
We also prove local universality at regular soft edges for droplets with several 
components, and lastly for soft/hard edges where a hard edge perfectly aligns 
with the droplet boundary. The main ingredients are Paley-Wiener type spectral 
embeddings for the Hilbert space associated with a limiting kernel, and the 
construction of weighted polynomials peaking near a given boundary point. 
\end{abstract}

\newgeometry{hmargin=3.2cm, vmargin={2.7cm, 3.3cm}}

\maketitle

\thispagestyle{empty}

\section{Introduction}
A central theme in random matrix theory, going back to the work of Wigner 
\cite{Wigner1958}, is the universality of local eigenvalue statistics.
For unitary Hermitian ensembles this phenomenon is well understood: local 
scaling limits in the bulk, at soft edges and at hard edges are governed 
respectively by the {\em sine}, {\em Airy} and {\em Bessel kernels}, see 
\cites{DG07,DKMVZ1999,TWAiry,TWBessel}. This holds independently of the 
choice of potential, provided that some mild regularity and non-criticality 
assumptions are met. The sine and Airy kernels appear also for Hermitian 
Wigner matrices with i.i.d.\ entries, see e.g.\ \cites{TaoVu2011,Soshnikov1999,
ErdosYauYin2012,EPRSY}. We refer to the surveys \cites{KuijlaarsUniversality,
DeiftUniversality,ErdosSurv} for additional background and perspectives.

This paper concerns the universality of local eigenvalue statistics for a class of 
non-Hermitian random matrices, specifically for the {\em normal matrix ensembles}.
These ensembles are natural non-Hermitian analogues of the classical unitary 
matrix ensembles, and just like in the Hermitian case the eigenvalues form 
Coulomb gases of charged particles. Universality of local scaling limits 
has been established in the bulk \cites{BermanBulk,AmeurHedenmalmMakarov2011} 
and at regular soft edges \cite{HedenmalmWennmanActa}. The roles of the sine 
and Airy kernels are played by the {\em Ginibre kernel} 
\begin{equation}
\label{eq:Gin}
G(\zeta,\eta)=e^{\zeta\bar{\eta}-\frac12(|\zeta|^2+|\eta|^2)}
\end{equation}
and the {\em Faddeeva kernel}, defined in terms of the complementary
error function \(\operatorname{erfc}(z)\) by
\begin{equation}
\label{eq:erfc}
H(\zeta,\eta)=\frac12\operatorname{erfc}\Big(\frac{\zeta+\bar{\eta}}{\sqrt{2}}\Big) 
\,e^{\zeta\bar{\eta}-\frac12(|\zeta|^2+|\eta|^2)}.
\end{equation}
To the best of our knowledge, these kernels first appeared as random matrix
scaling limits in \cite{Ginibre} and \cite{ForresterHonner}, respectively. 
We note in passing that the very same kernels \eqref{eq:Gin} and \eqref{eq:erfc} 
govern the local eigenvalue statistics of non-Hermitian Wigner-type random 
matrices as well \cites{TaoVu2015NonHerm,Erdos,CampbellCipolloniErdosJi2025}.

\newpage
\restoregeometry

The geometry of the support of the associated equilibrium measure (the \enquote{droplet})
is richer in two dimensions than in one, making the edge behavior rather more 
involved. Apart from a few case studies, the general theory is so far limited 
to connected droplets with analytic boundary, and only soft edges have been explored 
in generality. Edge universality has therefore remained open at hard edges, and also at 
soft edges for disconnected (\enquote{multi-cut}) droplets. The goal of this paper is 
to develop a unified approach that resolves these problems. A key feature of our 
approach is that it is {\em direct}: it altogether avoids orthogonal 
polynomials, which have proven difficult to analyze in these regimes. In the case 
of the real line, several powerful direct approaches have appeared, see, e.g., 
\cites{ALS,ELW,ELS,Lub,LevinLubinsky2008Varying} and the references therein. 
However, these rely essentially on the self-adjoint nature of the problem or 
apply only to fixed measures. In particular, the limiting kernels that arise in these
settings are reproducing kernels for de Branges spaces, see e.g.\ \cite{LubinskyDeBranges}.

Our main results (Theorems~\ref{thm:main-hard-edge}, \ref{thm:main-soft-edge} and 
\ref{thm:main-soft-hard-edge} below) settle the question of universality of local 
eigenvalue statistics at regular interfaces in the random normal matrix model, 
including the previously inaccessible hard edges. The proofs moreover give a 
conceptual description of the limiting kernels in terms of Hilbert spaces of 
entire functions. We assume throughout that the weights and the interfaces are 
real-analytic, but impose no symmetry conditions on either the potential 
or the interface.

\subsection{Interface transitions in the random normal matrix model}
We proceed to introduce the normal matrix model and explain how the various edge 
regimes arise. Fix a real-analytic function \(Q:\C\to \mathbb{R}\) 
(the \enquote{potential}) subject to the growth bound
\begin{equation}
\label{eq:growth}
\liminf_{|z|\to+\infty}\frac{Q(z)}{\log|z|^2}>1,
\end{equation}
and consider the random \(n\times n\)-matrix \(M\) picked from the probability 
measure
\begin{equation}
\label{eq:RNM}
\frac{1}{\mathcal{Z}_n}e^{-n\mathrm{Tr}(Q(M))}dM,
\end{equation}
where \(dM\) is the Riemannian volume form
on the (non-singular part of the) manifold of {\em normal} \(n\times n\)-matrices 
and \(\mathcal{Z}_n\) is a normalizing constant. A remarkable but well-known fact 
is that the eigenvalues of \(M\) follow a Coulomb gas law on the complex plane 
with external field \(Q(z)\) and inverse temperature \(\beta=2\). That is, the 
joint eigenvalue density with respect to Lebesgue measure on \(\C^n\) is given by
\begin{equation}
\label{eq:RNM-eigenvalues}
\frac{1}{Z_n}\prod_{1\le j<k\le n}|z_j-z_k|^2 \prod_{j=1}^n e^{-n Q(z_j)}
\end{equation}
where \(Z_n\) is a (different) normalizing constant. As \(n\) tends to infinity, 
the internal repulsive forces between the eigenvalues (or {\enquote{particles}}) are 
counterbalanced by the external confinement and as a result the gas condensates 
to a compact set called the {\em droplet}. To see what this set is, consider 
the weighted logarithmic energy functional
\begin{equation}
\label{eq:energy}
I_Q(\mu)=\int_{\C^2}\log\frac{1}{|z-w|}d\mu(z)d\mu(w)+\int_{\C}Q(z)d\mu(z).
\end{equation}
The functional \(I_Q(\mu)\) admits a unique minimizer \(\sigma=\sigma_Q\) among
all probability measures on \(\C\), which we refer to as the {\em (weighted) 
equilibrium measure}. By classical potential theory, the equilibrium measure 
is absolutely continuous with respect to normalized area measure 
\(dA(z)=\frac{1}{\pi} dxdy\) and its density is given by
\(\Delta Q(z)1_{{S}}\) for some compact set \({S}\), see 
\cites{SaffTotik,HedenmalmMakarovProcLondon} (our Laplacian 
\(\Delta=\partial\bar\partial\) is a quarter of the usual one).
We then have that
\[
\lim_{n\to\infty}\frac{1}{n}\sum_{j=1}^n f(z_j)= \int_{{S}}f(z)\Delta Q(z)dA(z)
\]
in probability for any bounded continuous function \(f\). In fact, much stronger 
estimates of large deviation type are known, see e.g.  
\cites{HedenmalmMakarovProcLondon,GZ-LDP,LS20017,BBNY19}.

A key feature of the eigenvalue process \eqref{eq:RNM-eigenvalues} is that it forms 
a determinantal point process. Its {\em correlation kernel}, denoted \(K_n(z,w)\), 
is the reproducing kernel for the space of weighted polynomials
\begin{equation}
\label{def:W-n}
\mathcal{W}_n=\Big\{p(z) e^{-nQ(z)/2}: p\in \C[z],\, \deg p\le n-1\Big\}
\end{equation}
endowed with the Hilbert space structure of \(L^2(\C,dA)\). Denoting by 
\(P_{0,n}, P_{1,n},\ldots, P_{n-1,n}\) the first \(n\) normalized orthogonal
polynomials in \(L^2(\C, e^{-nQ}dA)\) the kernel is given by
\begin{equation}
\label{def:K-n}
K_n(z,w)=\sum_{j=0}^{n-1}P_{j,n}(z)\overline{P_{j,n}(w)}\hspace{1pt}
e^{-\frac{n}{2}(Q(z)+Q(w))}.
\end{equation}
The condensation of the eigenvalues to the droplet is reflected by the weak convergence
of the one-point function
\[
\frac{1}{n}K_n(z,z)\,dA\to d\sigma.
\]

The interface \(\partial {S}\) is not
prescribed, but emerges naturally as a free boundary in the associated weighted potential 
theory problem. When \(Q\) is strictly subharmonic near the droplet boundary, 
there is a sharp density transition across the interface \(\partial {S}\) which is smooth if
we blow up at scale \(n^{-1/2}\). Such interfaces are referred to as {\em soft edges}. By now 
it is well understood that this soft-edge transition is governed by a universal scaling limit 
involving the complementary Gaussian error function. More precisely, if \(z_0\in\partial{S}\) 
is a regular boundary point with outward unit normal \(e^{i\theta}\), then for some {\em cocycles} 
\(c_n\) it should hold that
\[
\frac{1}{n\Delta Q(z_0)}\hspace{1pt} K_n\left(z_0+\frac{e^{i\theta}\zeta}{\sqrt{n\Delta Q(z_0)}},
z_0+\frac{e^{i\theta}\eta}{\sqrt{n\Delta Q(z_0)}}\right)c_n(\zeta,\eta)\to H(\zeta,\eta),
\] 
where \(H\) is the kernel in \eqref{eq:erfc}.
This was initially observed for the Ginibre ensemble \cite{ForresterHonner},
then for other special potentials for which the 
orthogonal polynomials can be analyzed using e.g.\ Riemann-Hilbert methods 
\cites{LeeRiser,BaloghBertolaLeeMcLaughlin} and under an priori translation-invariance 
assumption on the limiting kernel \cite{RescalingWardAKM}. A first general universality result 
for soft edges was obtained in \cite{HedenmalmWennmanActa} assuming that 
\(\partial{S}\) is a single, regular Jordan curve. This made use of a detailed
asymptotic expansion of the associated orthogonal polynomials
obtained in the same paper.

{\em Hard edges} are interfaces that arise when the eigenvalues are 
forbidden from occupying a given region \(\Omega\), see Figure~\ref{fig:hard-edge} 
for an illustration. The hard-edge constraint corresponds to orthogonal polynomials 
and reproducing kernels with respect to the truncated area measure 
\(1_{\C\setminus\Omega} e^{-nQ}dA\). This type of interface transition is less understood, 
and approaching it via orthogonal polynomials seems considerably more difficult than 
in the soft-edge case. Indeed, the local hard-edge scaling limits are known only 
under radial symmetry constraints \cites{Seo,AmeurCronvallHardWall,Shirai,ZyczkowskiSommers}
or for Coulomb gases associated with classical Bergman spaces with fixed weights 
\cites{LubinskyBergman, TotikChristoffelUniversality}. The reason for the difficulties 
appears to be that the orthogonal polynomials can have a very complicated behavior in 
the presence of hard edges. In particular this is so for varying exponential weights
and degrees \(k=\tau n\) such that \(\partial\Omega\) intersects the boundary of the 
associated \(\tau\)-droplet \(S_\tau\) (cf.\ \cite{HedenmalmWennmanActa}*{p.\ 316}).
Contrary to the bulk and soft-edge cases the natural scaling near a hard edge is \(n^{-1}\). 
The kernel appearing in \cite{LubinskyBergman} is
\begin{equation}
\label{eq:def-B}
B(\zeta,\eta)=\int_0^1 te^{-t(\zeta+\bar{\eta})}dt,
\end{equation}
and this should be a natural candidate for a universal
hard-edge kernel.

\subsection{Main results} 
\label{s:main-results}
\begin{figure}
\centering
\includegraphics[width=.8\textwidth]{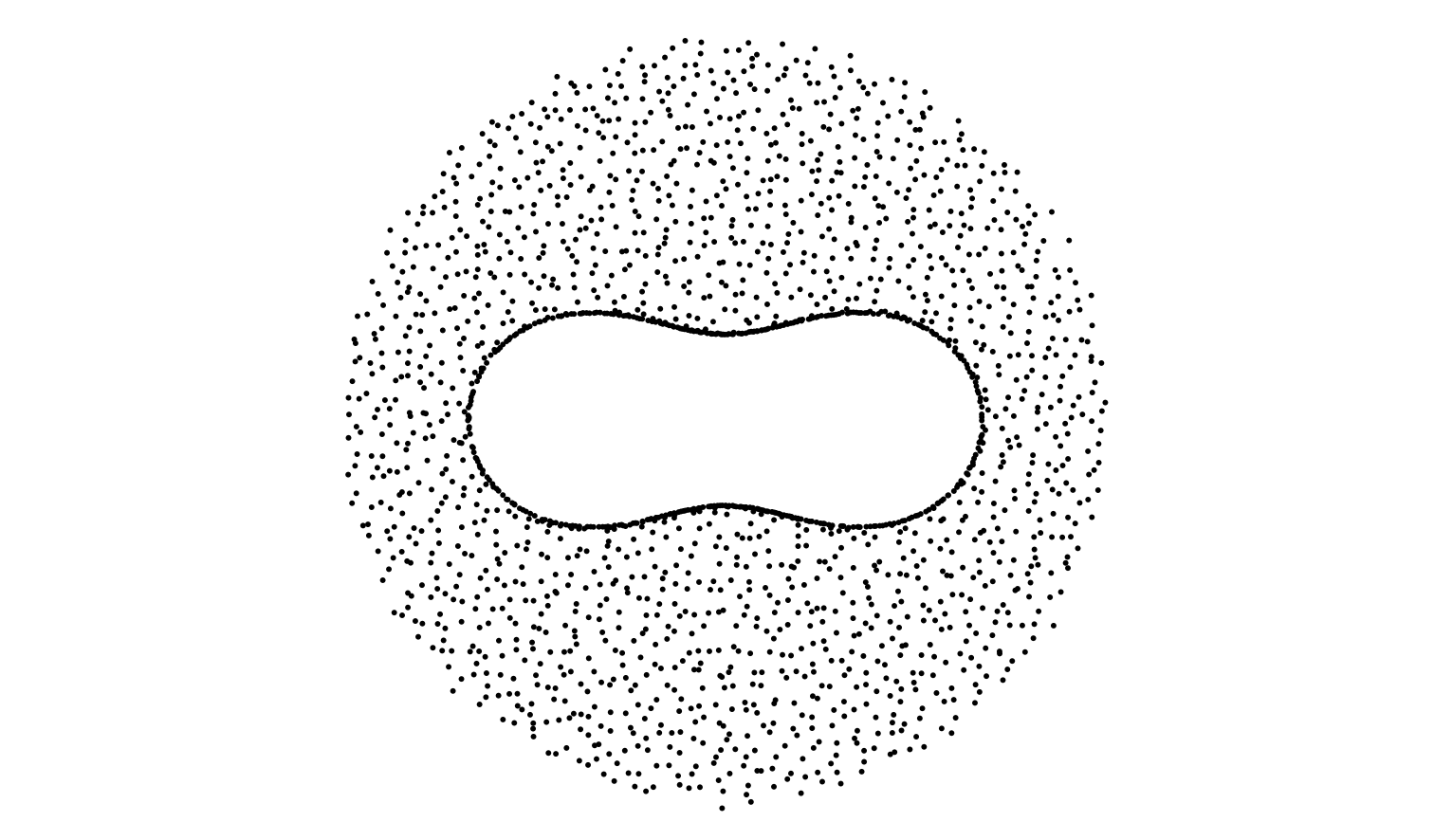}
\caption{Illustration of the Ginibre ensemble (\(n=2000\)) with 
a hard edge along an interior \enquote{hole} \(\Omega\). The outer edge
is a regular soft edge.}
\label{fig:hard-edge}
\end{figure}
We make the following standing assumptions. Let \(Q\) be a real-analytic potential 
subject to the logarithmic growth bound \eqref{eq:growth} and denote by 
\(S=\operatorname{supp} \sigma_Q\) the associated droplet. We assume throughout 
that \(Q\) is strictly subharmonic in a neighborhood of \(S\), and that
\(S\) consists of a finite number of finitely connected domains with analytic boundaries. 
Standard potential theoretic considerations imply that there exists a
constant \(F_\sigma\) for which we have \(Q+2U^{\sigma}\ge F_\sigma\) 
with equality on \(S=\mathrm{supp}(\sigma)\). We assume that {\em strict} 
inequality holds away from the droplet.

Denote by \(\Omega\) a Jordan domain with real-analytic boundary, and condition the 
Coulomb gas of eigenvalues \eqref{eq:RNM-eigenvalues} on the event that no particle
lies in \(\Omega\). In this way, we obtain a new determinantal point process whose 
correlation kernel \(K_n^\Omega(z,w)\) is the reproducing kernel for the space 
\(\mathcal{W}_n\) of weighted polynomials with the inner product of 
\(L^2(\C\setminus\Omega, dA)\). See Figure~\ref{fig:hard-edge} for an illustration.

Macroscopically, the gas is governed by a constrained equilibrium measure which minimizes the 
weighted logarithmic energy \(I_Q(\mu)\) from \eqref{eq:energy} over all probability measures 
with \(\mu(\Omega)=0\). The constrained equilibrium measure is known to take the form
\[
\sigma_\Omega = \sigma\vert_{{S}\setminus \Omega}
+\mathrm{Bal}(\sigma\vert_{\Omega}, \partial\Omega)  
\]
where \(\sigma=\sigma_Q\) is the unconstrained equilibrium measure and 
\(\mathrm{Bal}(\nu, \partial\Omega)\) denotes the {\em balayage} of a measure 
\(\nu\) on \(\Omega\) to \(\partial\Omega\), see Section~\ref{s:hard_wall} 
below and \cites{ASZ,AdhikariReddy,CharlierBalayage} for details.
The balayage measure is absolutely continuous with respect to 
arc-length on \(\partial\Omega\), and we denote its
density by \(\rho=\rho_\Omega\). For a point \(z_0\in \partial\Omega\), 
we denote by \(e^{i\theta}\) the outer unit normal to \(\partial\Omega\) 
at \(z_0\), pointing into \({S}\).

\begin{mainthm}[Hard edges]\label{thm:main-hard-edge}
Assume that \(\partial\Omega\) is a real-analytic Jordan curve, 
compactly contained in \(\operatorname{Int} (S)\) and let \(z_0\in\partial\Omega\).
Then there exists a cocycle \(c(\zeta,\eta)\) such that, locally uniformly for 
\(\zeta,\eta\in\C\), 
\[
\lim_{n\to\infty}\frac{1}{\rho(z_0)^2n^2}\hspace{1pt} 
K_{n}^\Omega\left(z_0+ \frac{e^{i\theta}\zeta}{\rho(z_0)n},
z_0+ \frac{e^{i\theta}\eta}{\rho(z_0)n}\right)c(\zeta,\eta)
=B(\zeta,\eta),
\]
where \(B\) is the Hermitian-entire kernel
\[
B(\zeta,\eta)=\int_0^1 te^{-t(\zeta+\bar{\eta})}dt.
\]
\end{mainthm}

In other words, the hard-edge kernels that appeared in 
\cites{LubinskyBergman} for classical Bergman spaces are universal 
also for hard-edge scaling limits in the normal matrix model, i.e., 
for varying exponential weights. The locally uniform convergence of
the rescaled kernels shows that the corresponding {\em rescaled determinantal point processes}
converge in law to the unique determinantal point field governed by the kernel 
\(B(\zeta,\eta)\).
The fact that the kernel depends
only on \(\zeta+\bar{\eta}\) shows that the limiting
point process at the edge is translation-invariant in 
the vertical (tangential) direction. 

The kernel \(B\) is the reproducing kernel for a rather curious space 
\(\mathcal{A}\) of entire functions. It is the closed subspace of the Bergman 
space on the right half-plane consisting of entire functions subject to the
growth constraint
\[
|f(\zeta)|\le C e^{|\re \zeta|},\qquad \re(\zeta)\le 0.
\]
The Bergman space on the right half-plane coincides with the image of
\(L^2(\R_+,\frac{dt}{t})\) under the Laplace transform \(\mathcal{L}\). 
In fact, \(\mathcal{L}\) is an isometric isomorphism between the two spaces. 
The subspace \(\mathcal{A}\) instead consists of all functions 
\(\mathcal{L}[F]\) with \(F\in L^2([0,1],\frac{dt}{t})\). 

The direct framework allows us to analyze a wide range of boundary 
confinements in a unified way. We also prove the following strengthening 
of the universality result in \cite{HedenmalmWennmanActa}
which allows for droplets with several components.

\begin{mainthm}[Soft edges]\label{thm:main-soft-edge}
Assume that \(\partial S\) consists of finitely many real-analytic Jordan 
curves, and let \(z_0\in\partial S\) with outer unit normal \(e^{i\theta}\). 
Then there exist cocycles \(c_n(\zeta,\eta)\) such that, locally uniformly 
for \(\zeta,\eta\in\C\), we have
\[
\lim_{n\to\infty}\frac{1}{n\Delta Q(z_0)}\hspace{1pt} 
K_{n}\left(z_0+ \frac{e^{i\theta}\zeta}{\sqrt{n\Delta Q(z_0)}},
z_0+ \frac{e^{i\theta}\eta}{\sqrt{n\Delta Q(z_0)}}\right)c_n(\zeta,\eta)
=H(\zeta,\eta)
\]
where \(H\) is the Faddeeva kernel \(\Phi(\zeta+\bar{\eta})\, G(\zeta,\eta)\), where
\begin{equation}
\label{eq:free-bdry-fcn}
\Phi(z)=\frac{1}{2}\mathrm{erfc}\Big(\frac{z}{\sqrt{2}}\Big)
=\frac{1}{\sqrt{2\pi}}\int_{z}^{\infty} e^{-t^2/2}dt.
\end{equation}
\end{mainthm}

The holomorphic Faddeeva kernel \({\bf H}(\zeta,\eta)=\Phi(\zeta+\bar{\eta})\,e^{\zeta\bar{\eta}}\) 
is the reproducing kernel for a subspace \(\mathcal{H}=\mathcal{H}_{\C}\) of the Fock space
consisting of functions \(f(\zeta)\) subject to the growth bound
\[
|f(\zeta)e^{\zeta^2/2}|\le C,\qquad \re(\zeta)\ge 0,
\] 
see \cites{HaimiHedenmalm,RescalingWardAKM}. The Fock space is the isometric image of \(L^2(\R)\)
under the Bargmann transform \(\mathcal{B}\) and, in analogy with the hard-edge space, the subspace
\(\mathcal{H}\) consists of functions with restricted Fourier support. More
precisely we have that
\(\mathcal{H}=\mathcal{B}(L^2(\R_{-}, dt))\).

We return to the Coulomb gas conditioned on the event that 
no particle lies in \(\Omega\).
The case when \(\Omega\) is precisely aligned with the 
droplet, i.e., \(\Omega=S^c\), is a critical case which sits
between the soft and the hard-edge cases. We refer 
to this regime as the {\em soft/hard edge regime}. 
This appears to have been introduced by Jancovici \cite{Jancovici}, 
and several results on this model have
appeared recently, see \cites{AmeurHardEdge,AmeurCronvallHardWall,AKMW,Seo}. 
However, all results to date have either relied on radial symmetry 
or an a priori assumption of tangential translation-invariance of the limiting kernel. 

\begin{mainthm}[Soft/hard edges]\label{thm:main-soft-hard-edge}
Assume that \(\partial S\) consists of finitely many real-analytic Jordan curves,
let \(\Omega=S^c\), and fix \(z_0\in\partial S\). Then there exist cocycles 
\(c_n(\zeta,\eta)\) such that, locally uniformly for \(\zeta,\eta\in\C\), we have
\[
\lim_{n\to\infty}\frac{1}{n\Delta Q(z_0)}\hspace{1pt} K_{n}^\Omega
\left(z_0+ \frac{e^{i\theta}\zeta}{\sqrt{n\Delta Q(z_0)}},
z_0+ \frac{e^{i\theta}\eta}{\sqrt{n\Delta Q(z_0)}}\right)c_n(\zeta,\eta)
=H_{\mathbb{L}}(\zeta,\eta),
\]
where \(H_{\mathbb{L}}\) is the kernel \(\Psi(\zeta+\bar{\eta})\,G(\zeta,\eta)\), where
\begin{equation}
\label{eq:soft/hard-bdry-fcn}
\Psi(z)=\frac{1}{\sqrt{2\pi}}\int_{-\infty}^0 \frac{e^{-(t-z)^2/2}}{\Phi(t)}dt,
\end{equation}
and \(\Phi(t)\) is the free boundary function \eqref{eq:free-bdry-fcn}.
\end{mainthm}

The limiting kernel is reproducing for a Hilbert space 
\(\mathcal{H}_{\mathbb{L}}\) of entire functions,
which coincides with the Faddeeva kernel space 
\(\mathcal{H}\subset F^2(\C)\) as sets but with the norm restricted to
the left half-plane \(\mathbb{L}\). 

\subsection{Outline of the main idea}

We illustrate the main new ideas in the soft-edge case.
The overall approach remains the same in the three cases, 
but the exact arguments differ between them.
Fix a point \(z_0\in\partial S\) and denote by \(e^{i\theta}\) 
the outward unit normal to \(S\). Put
\[
k_n(\zeta,\eta)=\frac{1}{n\Delta Q(z_0)}\hspace{1pt} 
K_n\left(z_0+\frac{e^{i\theta}\zeta}{\sqrt{n\Delta Q(z_0)}},
z_0+\frac{e^{i\theta}\eta}{\sqrt{n\Delta Q(z_0)}}\right).
\]
It was shown in \cite{RescalingWardAKM} that there 
exist cocycles \(\{c_n(\zeta,\eta)\}_n\) such that
\(\{c_n \cdot k_n\}_n\) is  a normal family. Any 
(subsequential) limit point has the structure
\[
{\bf K}(\zeta,\eta)\,e^{-\frac12(|\zeta|^2+|\eta|^2)}
\]
where \({\bf K}(\zeta,\eta)\) is a Hermitian-entire kernel 
function belonging to the Fock space \(F^2(\C)\) of square-integrable entire 
functions with respect to the standard Gaussian 
measure on \(\C\). 
What remains is to establish that \({\bf K}\) is the Faddeeva kernel
\[
{\bf H}(\zeta,\eta)=\Phi(\zeta+\bar{\eta})\, e^{\zeta\bar{\eta}}.
\]
Since Hermitian-entire kernels are determined by their diagonal restriction 
by polarization, it is enough to verify that \({\bf K}(\zeta,\zeta)={\bf H}(\zeta,\zeta)\),
which will be done by verifying the complementary inequalities 
\({\bf K}(\zeta,\zeta)\ge {\bf H}(\zeta,\zeta)\) and 
\({\bf K}(\zeta,\zeta)\le {\bf H}(\zeta,\zeta)\).

\bigskip

\noindent {\bf Step 1.} (Spectral embedding)\; A version of the 
Bernstein-Walsh inequality asserts that
\[
K_n(z,z)\le Cn e^{-n(Q-\check{Q})(z)},
\]
where \(\check{Q}(z)=-2U^{\sigma}(z)+F_\sigma\) for some constant \(F_\sigma\), 
chosen such that \(\check{Q}=Q\) on the droplet \(S\). This shows 
in particular that \(K_n(z,z)\) decays rapidly outside the droplet. 
By considering the local behavior of \(Q-\check{Q}\) near the 
boundary of \(S\), one gathers that 
\begin{equation}
\label{eq:pot-theory-bound}
|{\bf K}(\zeta,\eta)e^{\zeta^2/2}| \le C,\qquad \re(\zeta)\ge 0. 
\end{equation}
The key observation here is that \eqref{eq:pot-theory-bound} 
together with the Paley-Wiener-Schwartz theorem 
(Theorem~\ref{thm:PWS}) implies that the entire function 
\({\bf K}(\zeta,\eta)\,e^{\zeta^2/2}\) is the Fourier-Laplace 
transform of a distribution supported on the negative half-line. 

The Fock space is known to be the isometric image of 
\(L^2(\R,dt)\) under the Bargmann transform, defined by 
\[
\mathcal{B}[F](\zeta)=\frac{1}{(2\pi)^{1/4}}\int_{\mathbb{R}}e^{\zeta t}
e^{-\zeta^2/2}e^{-t^2/4}F(t)dt.
\]
Moreover, the Faddeeva kernel \({\bf H}(\zeta,\eta)\)
is the reproducing kernel for the subspace 
\[
\mathcal{H}=\mathcal{B}(L^2(\R_-,dt))
\] 
of \(F^2(\C)\), see \cite{HaimiHedenmalm}*{Proposition~5.4}.
An inspection reveals that the Bargmann transform is a Fourier-Laplace transform 
followed by multiplication by \(e^{-\zeta^2/2}\). 
We thus have two representations of \({\bf K}(\zeta,\eta)\,e^{\zeta^2/2}\) as 
Fourier transforms, and by Fourier inversion they have to agree. We may thus 
conclude that \(F\) is in fact supported on the negative half-line, 
and hence \({\bf K}(\cdot, \eta)\in\mathcal{H}\). This is what we refer to as 
the {\em spectral embedding}.

An argument combining the reproducing property of \(K_n(z,w)\) with 
Fatou's lemma shows that the limiting kernel \({\bf K}(\cdot, \eta)\) satisfies the 
norm-bound
\begin{equation}
\label{eq:upper-bd-L2-limKer}
\int_{\C}|{\bf K}(\zeta,\eta)|^2 e^{-|\zeta|^2}dA(\zeta)\le {\bf K}(\eta,\eta).
\end{equation}
Now, the reproducing kernel \({\bf H}\) of \(\mathcal{H}\) is {\em extremal}
in the sense that
\[
{\bf H}(\eta,\eta)\ge \frac{|f(\eta)|^2}{\int_{\C}|f(\zeta)|^2 
e^{-|\zeta|^2}dA(\zeta)},\qquad f\in\mathcal{H}.
\]
Applying this with \(f={\bf K}(\cdot,\eta)\in\mathcal{H}\) 
together with the norm-bound \eqref{eq:upper-bd-L2-limKer},
we obtain the desired inequality
\[
{\bf H}(\eta,\eta)\ge \frac{{\bf K}(\eta,\eta)^2}{\int_{\C}|{\bf K}(\zeta,\eta)|^2
e^{-|\zeta|^2}dA(\zeta)}\ge {\bf K}(\eta,\eta).
\]

Step 1 shares some similarities with the
approach of Levin and Lubinsky in \cite{LevinLubinsky2008Varying}.
However, the function theory is quite different without the
symmetries of the real line. The compactness argument and the bound 
\eqref{eq:upper-bd-L2-limKer} are taken from \cite{RescalingWardAKM} 
in the soft-edge case.

\bigskip

\noindent {\bf Step 2.} (Construction of boundary peak functions)\; The complementary 
lower bound will again be based on the extremal property of the reproducing kernel, 
but this time applied for finite \(n\). That is, we will use that
\begin{equation}
\label{eq:extremal-finite-n}
K_n(z,z)\ge \frac{|p(z)|^2 e^{-nQ(z)}}{\int_{\C}|p(w)|^2 e^{-nQ(w)}dA(w)}
\end{equation}
whenever \(p\) is a polynomial of degree at most \(n-1\). We will show how to construct good
{\em trial kernels} \(K_n^\sharp(\cdot,w)\) lying exponentially close to \(\mathcal{W}_n\),
satisfying the following properties:
\begin{enumerate}[leftmargin=.8cm]
\item[(i)] The trial kernel approximates the desired error 
function kernel \({\bf H}\): for \(\zeta\) fixed,
\[
\frac{1}{n\Delta Q(z_0)}K_n^\sharp\left(z_0+\frac{e^{i\theta}\zeta}{\sqrt{n\Delta Q(z_0)}},
z_0+\frac{e^{i\theta}\zeta}{\sqrt{n\Delta Q(z_0)}}\right)
={\bf H}(\zeta,\zeta)e^{-|\zeta|^2}+o(1).
\]
\item[(ii)] With \(w=z_0+e^{i\theta}\eta/\sqrt{n\Delta Q(z_0)}\) and \(\eta\) fixed, 
\[
\limsup_{n\to\infty} \frac{1}{n\Delta Q(z_0)}
\int_{\C}|K_n^\sharp (z,w)|^2 dA(w)\le {\bf H}(\zeta,\zeta)e^{-|\zeta|^2}.
\]
\end{enumerate}

\noindent Together with the extremal property of \(K_n\), these properties imply that
\begin{align}
\frac{1}{n\Delta Q(z_0)}\hspace{1pt} K_n(z,z)&\ge (1+o(1))\frac{{\bf H}(\zeta,\zeta)^2 
e^{-2|\zeta|^2}}{\frac{1}{n\Delta Q(z_0)}\int_{\C}|K_n^\sharp (z,w)|^2 dA(w)}
\\  &\ge (1+o(1)){\bf H}(\zeta,\zeta)e^{-|\zeta|^2}.
\end{align}
Combining this with the convergence of \(\frac{1}{n\Delta Q(z_0)}K_n(z,z)\) 
to \({\bf K}(\zeta,\zeta)e^{-|\zeta|^2}\)
the inequality
\[
{\bf K}(\zeta,\zeta) \ge {\bf H}(\zeta,\zeta)
\]
follows. 

\bigskip

We comment briefly on a problem that arises in the construction of \(K_n^\sharp(z,w)\), 
cf.\ Section~\ref{s:hard-edge-constr} and Section~\ref{s:soft-edge-constr} 
below. Property (ii) requires the kernels to be 
concentrated close to the point \(w\), making them {\em boundary analogues}
of Tian's classical peak sections, cf.\ \cite{Tian}. The construction of peak sections typically 
uses localization with \(\bar\partial\)-estimates, which break down at boundary points.
We instead construct the trial kernels as sums 
\[
K_n^\sharp(z,w)=\sum_{j=m_n}^{M_n}q_{j,n}(z)\overline{q_{j,n}(w)}e^{-\frac{n}{2}(Q(z)+Q(w))},
\]
where the individual wave functions \(q_{j,n}e^{-nQ(z)/2}\) peak along
certain curves close to the interface and have rapidly oscillating arguments.
The oscillations cause cancellations in the sum for \(z\ne w\), resulting in the off-diagonal 
decay of \(|K_n^\sharp(z,w)|^2\) needed for Property (ii). This type of cancellation was 
observed in \cite{AmeurCronvall} for the actual kernel \(K_n(z,w)\), where it was used to derive
off-diagonal asymptotics for \(z,w\) near \(\partial S\).

One recurring problem in our construction of the polynomials \(q_{k,n}\) is the 
following. We are given a domain \(\Omega\), either bounded or unbounded, whose 
boundary is a real-analytic Jordan curve in \(S\). For definiteness, let's consider 
the unbounded case. To carry out our construction, one would naively like to find a 
function \(u\) of prescribed logarithmic growth such that
\begin{equation}\label{eq:free-bdry}
\Delta u=0\;\text{ on }\; \Omega\qquad\text{ and }
\qquad (u-Q)=\nabla(u-H)=0 \;\text{ on }\;\partial\Omega,
\end{equation}
where \(H\) is a given real-analytic function. The function \(u\)
would play a role similar to the function \(\breve{Q}_\tau\) in the construction of
the \enquote{quasipolynomials} in \cite{HedenmalmWennmanActa}*{Section~4.1}.
However, \eqref{eq:free-bdry} is over-determined and can 
only be solved for very particular choices of 
\(\Omega\) and \(H\). (For instance, if \(H=Q\) 
and \(\Omega\) is the complement of a \(\tau\)-droplet
then we get that \(u=\breve{Q}_\tau\).)
To overcome this, we can give ourselves some additional freedom by instead considering
\[
\Delta u=\nu \quad\text{on }\Omega,\qquad (u-Q)=\nabla(u-H)=0\text{ on }\partial\Omega,
\]
where \(\nu\) is an absolutely continuous measure with smooth positive density on
a Jordan curve \(\Lambda\subset \Omega\). Choosing \(\Lambda\) and \(\nu\) appropriately,
the problem admits a solution  \(u\).
This subharmonic function is no longer the real part of 
a holomorphic function in \(\Omega\), but by discretizing
its Riesz mass we can still find a holomorphic 
function \(f\), having a \enquote{{\em necklace}} of tightly packed
zeros along \(\Lambda\),
such that \(|f|^2e^{-nQ}\sim e^{-n(Q-u)}\) with
acceptable error.

In the hard-edge case the construction is used to find weighted polynomials
which peak along \(\partial\Omega\) with a 
prescribed decay rate in the normal direction into \(S\). 
In the soft and soft/hard-edge cases, we use the same construction to find 
weighted polynomials which peak only along one of the components of \(\partial S\).

We stress that the polynomials used in the construction of \(K_n^\sharp(z,w)\)
have nothing to do with the orthogonal polynomials\footnote{A somewhat similar 
flexibility was used in \cites{BGZNW,HedenmalmWennmanCMP}.}, 
and in fact do not satisfy any orthogonality conditions.

\section{Hard edges}
\label{s:hard_wall}
\noindent Let $Q$ be a real-analytic potential with droplet $S$, 
subject to the standing assumptions of Section~\ref{s:main-results}. 
We take a simply connected domain \(\Omega\subset\widehat{\C}\) (bounded or unbounded) 
with real-analytic boundary $\Gamma = \partial \Omega$ compactly 
contained in $\operatorname{Int} (S)$. We consider the restricted 
Coulomb gas where the particles are forbidden from being in the domain $\Omega$. 
There are two cases: either \(\Omega\) is bounded and compactly contained in \(S\), 
or \(\Omega\) is unbounded and the complement \(S^c\) is contained
in \(\Omega\). We will give the proof of Theorem~\ref{thm:main-hard-edge} 
in the case when \(\Omega\) is bounded with \(\Omega\subset S\). The modifications 
needed for the unbounded case will hopefully be clear from the soft-edge discussion 
which deals with that setting explicitly.

The correlation kernel $K_n^{\Omega}(z,w)$ is given by
\[
K_n^{\Omega}(z,w) = \sum\limits_{j=0}^{n-1} P_{j,n}^\Omega(z)
\overline{P_{j,n}^\Omega(w)} e^{-\frac{n}{2}(Q(z)+Q(w))},
\]
where $\{P_{k,n}^\Omega\}_{k=0}^{n-1}$ are the orthonormal 
polynomials with respect to the weighted area 
measure $1_{\C\setminus \Omega}(z) e^{-nQ(z)}dA(z)$. 
In other words $K_n^{\Omega}$ is the reproducing kernel 
of the space $\mathcal{W}_n$ of weighted polynomials as 
a subspace of $L^2(\C\setminus \Omega,dA)$.
When no confusion should arise, we sometimes omit the 
superscript \(\Omega\) and write \(K_n(z,w)\) for the kernel.

The following result about the restricted equilibrium potential 
is well known from potential theory.
\begin{thm}
The equilibrium measure $\sigma_\Omega$ in potential $Q$ restricted to $\C\setminus \Omega$ 
is given by 
\begin{equation}\label{structure}
d\sigma_\Omega(z) = \Delta Q(z) 1_{S\setminus \Omega}(z) dA(z) + d\mu(z),
\end{equation}
where $\mu$ is the Balayage measure of $\Delta Q(z) 1_\Omega(z) dA(z)$ onto $\Gamma$. 
Moreover, $\mu$ is absolutely continuous with respect to 
arc-length measure on $\Gamma$ with strictly positive density $\rho$.
\end{thm}

\begin{proof} Since $\sigma$ is the equilibrium measure on $\C$ in the potential $Q$ we have
\begin{align}
&2U^\sigma(z)+Q(z) = F_\sigma, \quad z\in \text{supp}\, \sigma,\\
&2U^\sigma(z) + Q(z)\geq F_\sigma, \quad z\in \C,
\end{align}
for a constant $F_\sigma$. From the definition of the Balayage measure we have
$$
U^{\sigma}(z) = U^{\sigma_\Omega}(z), \quad z\in \C\setminus \Omega,
$$
and 
$$
U^{\sigma}(z) \geq U^{\sigma_\Omega}(z), \quad z\in \C.
$$
We get that
\begin{align*}
2U^{\sigma_\Omega}(z) + Q(z) &=F_\sigma, \quad z\in \text{supp}\, \sigma\setminus \Omega,\\
2U^{\sigma_\Omega}(z) + Q(z) &\geq F_\sigma, \quad z\in \C,
\end{align*}
which by \cite{SaffTotik}*{Theorem 3.3} implies that $\sigma_\Omega$ is the restricted 
equilibrium measure in the potential $Q$ on $\C\setminus \Omega$.  
    
To prove that the density $\rho$ is strictly positive we use a representation of the 
Balayage measure in terms of the Green's function $g_\Omega(z,\zeta)$ on $\Omega$. 
The Green's function on $\Omega$ is the harmonic function in $\Omega\setminus \{\zeta\}$ 
such that $g_\Omega(z,\zeta) = 0$ for $z\in \partial \Omega$ and $g_\Omega(z,\zeta) 
= \log \frac{1}{|z-\zeta|} + \bigO(1)$ as $z\to \zeta$.

For a measure $\nu$ on $\Omega$ we can represent the Balayage measure $\widehat{\nu}$ 
of $\nu$ onto $\Gamma$ as
$$
\frac{d\widehat{\nu}}{ds}(z) = -\int \partial_{\boldsymbol{n}} g_\Omega(z,\zeta) \,d\nu(\zeta),
$$
where $\bf{n}$ is the outward unit normal of $\Omega$ and $ds$ denotes the 
arc-length measure on \(\Gamma\) 
normalized by \(\frac1{2\pi}\).

The curve $\Gamma$ is real-analytic so by the Hopf lemma the normal derivative of the 
Green's function is strictly negative on $\Gamma$. In our case $d\nu = 1_{\Omega} \Delta Q dA$. 
By assumption $Q$ is strictly subharmonic on $\Omega$ and we get that the density $\rho$ is also 
strictly positive.
\end{proof}

There is a constant $F_{\sigma_\Omega}$ such that the function
\[
\check{Q}(z) = F_{\sigma_\Omega} - 2U^{\sigma_\Omega}(z), 
\] 
is the largest subharmonic minorant of $Q$ on $\C\setminus \Omega$ with growth at most 
$\log |z|^2 + \bigO(1)$ as $|z|\to \infty$. The function $\check{Q}$ is harmonic in $\Omega$
and we write $V$ for the harmonic extension of $\check{Q}$ across $\Gamma$ and $\mathcal{V}$ 
for the analytic function with real part $V$ which can also be defined in a neighborhood of $\Omega$. 
The imaginary part is uniquely defined up to a constant, which will not play any role for us.

We can recover the density $\rho$ from the function $V$.
\begin{lem}
\label{eq:lem-normal-deriv-density} Let $\partial_{\bf{n}}$ denote the normal 
derivative in the outward direction of $\Omega$. We have that
$$
\partial_{\bf{n}} (V(z)-Q(z)) = -2\rho(z), \quad z\in \Gamma.
$$ 
\end{lem}

\begin{proof}
From basic potential theory it follows that
$$
\rho(z) = -\Big(\frac{\partial U^{\sigma_\Omega}}{\partial \bf{n}_e}(z) 
+  \frac{\partial U^{\sigma_\Omega}}{\partial \bf{n}_i}(z) \Big),
$$
where $\bf{n}_e$ and  $\bf{n}_i$ are the exterior and interior unit normals of $\Omega$, 
respectively. The lemma now follows from the identities
 \[
\frac{\partial U^{\sigma_\Omega}}{\partial \bf{n}_i}(z) 
= -\frac{1}{2}\frac{\partial V}{\partial \bf{n}_i}(z),
\quad \quad \frac{\partial U^{\sigma_\Omega}}{\partial \bf{n}_e}(z) 
= -\frac{1}{2}\frac{\partial Q}{\partial \bf{n}_e}(z). \qedhere
\]
\end{proof}

\subsection{Limiting kernels and the upper bound}
We need some properties of the correlation kernel $K_n(z,w)=K_n^\Omega(z,w)$.

\begin{lem}\label{lem:apriori_hard}
For $z\in \C$ we have
\begin{equation}\label{exp_bound_global}
K_n(z,z) \leq Cn^2 e^{n(\check{Q}(z)-Q(z))}.
\end{equation}
In particular, the kernel $K_n(z,z)$ satisfies 
$$
K_n(z,z)\leq C n^2, 
$$
for $z\in\C\setminus \Omega$.
\end{lem}

\begin{rem}
The lemma only gives a good estimate close to the domain $\Omega$. 
In fact for any interior point of $\C\setminus \Omega$ 
with ${\rm d}(z,\Omega)\ge \delta n^{-1/2}\log n$, $\delta>0$, we have the stronger estimate
$$
K_n(z,z) \leq C n,
$$
with possibly a different constant.
\end{rem}

\begin{proof}     
Take a curve $\widetilde{\Gamma}$ in $\C\setminus \Omega$ at distance $\frac{1}{n}$ 
from $\Gamma$. Around any point on $\widetilde{\Gamma}$ we can then take a disc of 
radius proportional to $\frac{1}{n}$ in $\C\setminus \Omega$. 
A standard subharmonicity estimate (cf.\ \cite{RescalingWardAKM}*{Lemma 3.1}) then 
gives that $K_n(z,z)\leq Cn^2$ on $\widetilde{\Gamma}$.

Replace $Q$ by $Q_h$, the harmonic function on the domain $\operatorname{Int}(\widetilde{\Gamma})$
interior to \(\widetilde{\Gamma}\) which equals $Q$ 
on $\widetilde{\Gamma}$. We then have that
$$
\log K_n(z,z)  + nQ(z) - nQ_h(z),
$$ 
is a subharmonic function bounded by $\log(Cn^2)$ on the curve $\widetilde{\Gamma}$. 
From the maximum principle we get that the same bound holds on $\textrm{Int}(\widetilde{\Gamma})$ 
and in particular on $\Gamma$. It only remains to estimate $n(Q-Q_h)$. Since the distance between 
the two curves is only $\frac{1}{n}$ and since \(Q-Q_h\) is smooth, this will only contribute with a constant.  
    
The second part of the theorem is a direct consequence of the definition of $\check{Q}$. 
\end{proof}
We consider the rescaled kernel
$$
k_n(\zeta,\eta) = \frac{1}{n^2\rho(z_0)^2}\hspace{1pt} K_n(z_n, w_n),
$$
where the points $z_n$ and $w_n$ are defined by
\begin{equation}
\label{eq:rescaling-points-hard-edge}
z_n = z_0 + \frac{e^{i\theta}\zeta}{n\rho(z_0)}, \quad w_n = z_0 + \frac{e^{i\theta}\eta}{n\rho(z_0)}.
\end{equation}
As in the statement of Theorem \ref{thm:main-hard-edge} we let $e^{i\theta}$ 
denote the outer unit normal to $\Omega$ at $z_0$.

Lemma \ref{lem:apriori_hard} contains all the information we need about $K_n(z,w)$ to
prove the following compactness result for the sequence of rescaled kernels. It gives 
direct hard-edge analogues of the {\em compactness}, 
{\em mass-one inequality} and {\em a priori estimates}
from \cite{RescalingWardAKM}.

\begin{thm}\label{thm:limiting_hard} There exists a cocycle \(c\) such that $(c\cdot k_n)_{n=0}^\infty$ 
is a normal family. Furthermore, any limit point $K(\zeta,\eta)$ is a Hermitian-entire function 
satisfying the \enquote{mass-one inequality}
$$
\int\limits_{\mathbb{H}} |K(\zeta,\eta)|^2 dA(\eta) \leq K(\zeta,\zeta)
$$
and
\begin{equation}\label{exp_bound_local}
|K(\zeta,\eta)| \leq C_\eta\, e^{|\re\, \zeta|}, \quad \re \zeta \leq 0.
\end{equation}
\end{thm}

\begin{proof} 
Define the function $A_n(\zeta,\eta)$ by
$$
A_n(\zeta,\eta) = e^{\frac{n}{2}(\mathcal{V}(z_n)+\overline{\mathcal{V}(w_n)})} 
e^{-\frac{n}{2}(Q(z_n)+Q(w_n))},
$$  
where \(\mathcal{V}\) is the holomorphic function defined above whose real
part equals \(V\) on a neighborhood of \(\Omega\).
Recall from Lemma~\ref{eq:lem-normal-deriv-density} that we have
$$
\partial_{\boldsymbol{n}}(V-Q)(z_0) = -2\rho(z_0),
$$
and hence
$$
\frac{n}{2}(\mathcal{V}(z_n) + \overline{\mathcal{V}(w_n)} -Q(z_n)-Q(w_n) )
= - (\zeta + \bar{\eta}) + ig(\zeta,\eta)+\bigO\Big(\frac{\|(\zeta,\eta)\|^2}{n^2}\Big),$$
where $g$ is the cocycle given by
$$
g(\zeta,\eta) = \partial_{\mathbf{n}} \im (\mathcal{V}(z_0)) \frac{\re (\zeta-  \eta)}{2\rho(z_0)} 
+  \partial_{\mathbf{n}} Q(z_0) \frac{\im (\zeta- \eta)}{2\rho(z_0)}.
$$
It follows that, with $c(\zeta,\eta) = e^{-ig(\zeta,\eta)}$, we have
$$
c(\zeta,\eta)A_n(\zeta,\eta) \to e^{-(\zeta+\bar{\eta})}, 
$$
locally uniformly as $n\to \infty$. Define the Hermitian-analytic function
$$
\Psi_n(\zeta,\eta) = \frac{k_n(\zeta,\eta)}{A_n(\zeta,\eta)}.
$$
We want to prove that $\{\Psi_n\}$ is a normal family from which the first part of the theorem 
would follow. From the reproducing property of $K_n$ we have 
$|k_n(\zeta,\eta)|^2\leq k_n(\zeta,\zeta)k_n(\eta,\eta)$.  We obtain
$$
|\Psi_n(\zeta,\eta)|^2 \leq C\frac{ k_n(\zeta,\zeta)k_n(\eta,\eta)}{|e^{-(\zeta+ \bar{\eta})}|^2},
$$
for $n$ big enough. Using the bound \eqref{exp_bound_global} from Lemma \ref{lem:apriori_hard} we 
get that $\Psi_n$ is locally uniformly bounded and hence a normal family. This proves the first 
part of the theorem and we see that the cocycle is given by $c(\zeta,\eta) = e^{-ig(\zeta,\eta)}$.

In order to prove the properties of the limiting kernel we consider a subsequence $n_k$ along 
which $\Psi_{n_k}\to \Psi$ for some Hermitian-entire function $\Psi$. It follows that along 
the same subsequence $c(\zeta,\eta)\cdot k_n(\zeta,\eta)$ converges locally uniformly to $K(\zeta,\eta)$. 

The reproducing property of $K_n(z,w)$ gives that for any $z\in \C$
$$
\int\limits_{\C\setminus \Omega} |K_{n_k}(z,w)|^2 dA(w) = K_{n_k}(z,z).
$$
We choose $z=z_{n_k}$ and make the change of variables 
$w = z_0 +\frac{e^{i\theta}\eta}{n_k\rho(z_0)}$. We get
$$
\int\limits_{H_{n_k}} |k_{n_k}(\zeta,\eta)|^2 dA(\eta) 
\leq \frac{1}{(n_k\rho(z_0))^2} K_{n_k}(z_{n_k},z_{n_k}),
$$
where $H_{n_k}$ is the image of $\C\setminus \Omega$ under the rescaling. 
The indicator function of $H_{n_k}$ converges to the indicator function 
of $\mathbb{H}$, hence from Fatou's lemma we have
$$
\int\limits_{\mathbb{H}} |K(\zeta,\eta)|^2 dA(\eta) \leq K(\zeta,\zeta).
$$  

This proves the mass-one inequality. In order to prove the exponential growth 
bound on the limiting kernel we make a Taylor expansion for $\re \zeta <0$,
\[
n(\check{Q}-Q)(z_n) = -2\re \zeta + \bigO\Big(\frac{1}{n}\Big).  
\]
From \eqref{exp_bound_global} we then get
$$
\frac{1}{\rho(z_0)^2n^2}K_n(z_n,z_n) \leq C e^{-2\re \zeta} \big(1+o(1)\big),
\quad \re \zeta \leq 0,
$$
as \(n\to \infty\).
From Cauchy-Schwarz we see that any limiting kernel $K$ must satisfy 
$$
|K(\zeta,\eta)| \leq C e^{-\re \zeta} , \quad \re \zeta \leq 0,
$$
which completes the proof.
\end{proof}

Theorem \ref{thm:limiting_hard} leads us to consider the subspace $\mathcal{A}$ 
of the Bergman space on the right half-plane consisting of entire functions 
with the growth restriction
$$
|f(\zeta)|\leq C e^{|\re \zeta|}, \quad \re(\zeta)\leq 0.
$$ 
This space has a clear interpretation in terms of the Fourier transform, and 
for this we will need a version of the classical Paley-Wiener-Schwartz theorem.
For a distribution \(u\in \mathcal{D}'\) such that \(e^{t\im(\zeta)} u(t)\in\mathcal{S}'\), 
we define the Fourier-Laplace transform at \(\zeta\) to be
\[
\widehat{u}(\zeta) =\mathcal{F}[u](\zeta)= u_x(e^{-ix\zeta}).
\]
The following result is a special case of Theorems~7.4.2 and 7.4.3 in 
\cite{HormanderVol1}. 

\begin{thm}\label{thm:PWS} If $U$ is an entire function that for some $c\geq 0$ satisfies
$$
|U(t + i\eta)| \leq C e^{c\eta}, \quad \eta\geq0,
$$
then there exists a distribution $u$ such that $e^{t \eta} u(t) \in \mathcal{S}'$ 
for every $\eta\geq0$ and such that \(u\) has Fourier-Laplace transform $U(t + i\eta)$ 
for every $\eta \geq 0$. In addition, the distribution $u$ is supported on $(-\infty,c]$.
\end{thm}

We now have the following result. Recall the definition of the kernel \(B\) from \eqref{eq:def-B}.

\begin{thm} 
\label{thm:space-A}
The function $B_\eta(\zeta) = B(\zeta,\eta)$ is the reproducing kernel of the the space $\mathcal{A}$.  
\end{thm}

\begin{proof}
Denote by $A^2(\mathbb{H})$ the Bergman space on the right half-plane, that is, the space of 
holomorphic functions in $\mathbb{H}$ with finite norm
$$
\|f\|^2_{A^2(\mathbb{H})} = \int\limits_{\mathbb{H}} |f(\zeta)|^2 dA(\zeta) <\infty.  
$$
The Laplace transform $\mathcal{L}$
$$
\mathcal{L}[F](\zeta) = \int\limits_{0}^{\infty} F(t) e^{-t\zeta} dt,
$$
is an isometric isomorphism $\mathcal{L}: L^2(\R_+, \frac{dt}{t}) \to A^2(\mathbb{H})$, 
cf.\ \cite{DGMR}. With the aid of Theorem~\ref{thm:PWS} applied in both the right and 
left half-planes, the growth restriction on the functions in $\mathcal{A}$ implies that
any \(f\in \mathcal{A}\) can be written as a Fourier-Laplace transform 
\[
f(\zeta)=u_x\left(e^{-x\zeta}\right)
\]
for some distribution \(u\)
supported on \([0,1]\) with \(e^{-\re(\zeta)t}u(t)\in \mathcal{S}'\) for 
\(\zeta\in\C\). Moreover, we have the representation \(f=\mathcal{L}[F]\) 
for some \(F\in L^2(\R_+,\frac{dt}{t})\), and we claim that \(F=u\) so that 
\(F\) is in fact supported on \([0,1]\). Indeed, since \(F\in L^2(\R_+,\frac{dt}{t})\)
it holds that \(e^{-t\re(\zeta)}F(t)\in \mathcal{S}'\) for \(\re(\zeta)> 0\) 
and hence the claim follows by Fourier inversion in \(\mathcal{S}'\). 

Since it is moreover clear that \(\mathcal{L}[F]\in \mathcal{A}\) for any 
\(F\in L^2([0,1],\frac{dt}{t})\), it follows that the Laplace transform is 
an isometric isomorphism from  $L^2([0,1],\frac{dt}{t})$ onto the space $\mathcal{A}$.

It remains to show that $B$ is the reproducing kernel 
of $\mathcal{A}= \mathcal{L}[L^2([0,1],\frac{dt}{t})]$. 
Let $L_\eta(\zeta)$ be the reproducing kernel of $\mathcal{A}$ and consider 
$F\in L^2([0,1], \frac{dt}{t})$. From the reproducing property of $L$ we get
$$
\big\langle \mathcal{L}[F] , B_\eta \big\rangle_{A^2(\mathbb{H})} 
= \big\langle F , \mathcal{L}^{-1} [B_\eta] \big\rangle_{L^2(\R,\frac{dt}{t})},
$$
from which we see that
$$
\int\limits_{0}^1 F(t)e^{-t\eta} dt = \int\limits_{0}^1 F(t) 
\overline{\mathcal{L}^{-1} [L_\eta](t)}\, \frac{dt}{t},
$$
for all $F\in L^2([0,1], \frac{dt}{t})$. We can conclude that 
$\mathcal{L}^{-1} [L_\eta](t) = 1_{[0,1]}(t)te^{-t\bar{\eta}}$. 
Applying the transform $\mathcal{L}$ to both sides gives
\[
L(\zeta,\eta) = \int\limits_{0}^1 t e^{-t(\zeta + \bar{\eta})} dt =B(\zeta,\eta). \qedhere
\] 
\end{proof}

From Theorem \ref{thm:limiting_hard} we get the following corollary.
\begin{cor}\label{cor:upper_bound_hard} Any limiting kernel $K_\eta(\zeta)$ as in 
Theorem \ref{thm:limiting_hard} belongs to the space $\mathcal{A}$ and satisfies the inequality
$$
K(\zeta,\zeta) \leq B(\zeta,\zeta).
$$
\end{cor}

\begin{proof} Since $B$ is the reproducing kernel for \(\mathcal{A}\) we have
\begin{equation}
\label{eq:extremal-B}
B(\zeta,\zeta) = \sup\limits_{f\in\mathcal{A}\setminus \{0\}} \frac{|f(\zeta)|^2}{\|f\|_\mathcal{A}^2}
\end{equation}
and since from Theorem \ref{thm:limiting_hard} we have 
$\|K_\zeta\|_\mathcal{A}^2 \leq K(\zeta,\zeta)$ we get the inequality
by applying the extremal property \eqref{eq:extremal-B} with \(f=K(\cdot,\zeta)\). 
\end{proof}

\subsection{Approximate kernel and the lower bound}
\label{s:hard-edge-constr}
\noindent The goal of this section is to construct polynomial trial kernels $K_n^{\sharp}(z,w)$ 
with good off-diagonal decay, and whose rescalings approximate \(B(\zeta,\eta)\) locally. 
In order to do so we first need to define some functions which will be used in the construction.

Recall that the constrained equilibrium measure \(\sigma_\Omega\) takes the form
\[
\sigma_\Omega=\sigma\vert_{\C\setminus \Omega} + \mu,
\]
where \(\mu=\mathrm{Bal}(\sigma\vert_{\Omega},\partial\Omega)\) is absolutely 
continuous with respect to arc-length measure on \(\partial\Omega\). 
Its density, denoted by \(\rho\), is real-analytic and strictly positive.
We begin with an inverse Balayage construction which moves the measure $\mu$ into $\Omega$. 
More precisely, we want to construct a measure $\nu$ on a real-analytic Jordan curve $\Lambda$ 
inside $\Omega$ such that $\widehat{\nu} = \mu$. We have the following theorem.

\begin{lem} 
\label{lem:balayage-hard-edge}
There exists a real-analytic Jordan curve $\Lambda$ inside $\Omega$ and a measure 
$\nu$ with a smooth density on $\Gamma$ such that $\widehat{\nu} = \mu$.  
\end{lem}
\begin{proof}
The potential $U^{\mu}$ is harmonic in the exterior of $\Gamma$ and continuous up to the boundary. 
We can extend it to a harmonic function $H$ in a neighborhood of $\textrm{Ext}(\Gamma)$. Now take a 
real-analytic Jordan curve $\Lambda$ close enough to $\Gamma$ so that $H$ is defined there. 
Define a new function
$$
\widetilde{H}(z) = 
\begin{cases} H(z), &\quad z\in \textrm{Ext } \Lambda,\\
P_\Lambda[H](z), &\quad z\in \textrm{Int }\Lambda,
\end{cases}
$$
where $P_{\Lambda}[H]$ denotes the harmonic function with boundary values $H$ on $\Lambda$.

Define a measure $\nu$ on $\Lambda$ by 
$$
d\nu(z) =  -\Big[\frac{\partial \widetilde{H}}{\partial \bf{n}_e}(z) 
+ \frac{\partial \widetilde{H}}{\partial \bf{n}_i}(z)\Big] ds(z).
$$
By taking \(\Lambda\) close enough to \(\Gamma\) we can ensure that \(\nu\) is a positive measure. 
Indeed, if we were to take \(\Lambda=\Gamma\) then we would have \(\widetilde{H}=U^\mu\), which 
has a strictly positive Neumann jump \(-(\partial_{\bf{n}_e}+\partial_{\bf{n}_i})U^\mu\) 
across \(\Gamma\). If we move \(\Lambda\) slightly 
inside \(\Omega\), the Neumann jump of \(\widetilde{H}\) clearly remains positive.
The potential $U^{\nu}=\widetilde{H}$ will be equal to $H$ on the exterior of $\Lambda$ and 
hence equal to $U^{\mu}$ on the exterior of $\Gamma$. We get that $\widehat{\nu} = \mu$. 
\end{proof}

The function 
$$
U_{\Omega}^{\nu}(z) = \int\limits_{\Lambda} g_\Omega(z,\zeta) d\nu(\zeta),
$$
is the so called Green's potential in $\Omega$ of the measure $\nu$.

\begin{lem}\label{greens_potential} The function $U^{\nu}_{\Omega}$ is zero on $\Gamma$ and 
$\partial_{\boldsymbol{n}_i} U^{\nu}_{\Omega}(z) = \rho(z)$, for $z\in \Gamma$.
\end{lem}

\begin{proof} That $U^{\nu}_{\Omega}$ is zero on $\Gamma$ follows directly since 
$g_\Omega(z,\zeta)$ is zero for $z\in \Gamma$. To get the expression for the normal 
derivative we first write
$$
U_\Omega^{\nu}(z) = U^{\nu}(z) - U^{\widehat{\nu}}(z).
$$
We get that for $z\in \Gamma$
\begin{align*}
\partial_{\boldsymbol{n}_i} U_\Omega^{\nu}(z) &= \partial_{\boldsymbol{n}_i} U^{\nu}(z) 
- \partial_{\boldsymbol{n}_i} U^{\mu}(z)\\
& = \big(\partial_{\boldsymbol{n}_i} U^{\nu}(z) + \partial_{\boldsymbol{n}_e} U^{\nu}(z) \big) 
- \big(\partial_{\boldsymbol{n}_i} U^{\mu}(z) + \partial_{\boldsymbol{n}_e} U^{\mu}(z) \big)\\
&= \rho(z),
\end{align*}
where we have used that $U^{\nu}(z)=U^{\mu}(z)$ in the exterior of $\Gamma$.    
\end{proof}

We define the probability measure $\lambda = \frac{\nu}{\nu(\Lambda)}$. We want to 
approximate $j\cdot\lambda$ by $j$ point masses on $\Lambda$. More precisely, 
we want to approximate the function $U_\Omega^{\lambda}$ by a function of the form
$$
U_\Omega^{\lambda_j}(z) := \frac{1}{j}\sum\limits_{k=1}^j g_\Omega(z,\zeta_{k,j}),
$$
where $\{\zeta_{k,j}\}_{k=1}^j$ are points on the curve $\Lambda$.

\begin{lem}\label{point_approx} There exist points $\{\zeta_{k,j}\}_{k=1}^j$ on 
$\Lambda$ and a positive constant $c>0$ such that
$$
U_\Omega^\lambda(z)- U_\Omega^{\lambda_j}(z)= \bigO(e^{-cj}),
$$
as $j\to \infty$, uniformly for $z$ in a neighborhood of $\Gamma$. 
\end{lem}

\begin{proof} There exists a neighborhood $N_1$ of $\Gamma$ and a neighborhood $N_2$ of $\Lambda$ 
such that $g_\Omega(z,\zeta)$ is real-analytic on $N_1\times N_2$.

The curve $\Lambda$ is real-analytic and the measure $\lambda$ has a real-analytic density 
\(\varpi\) with respect to arc-length measure on $\Lambda$. Let \(\gamma\) denote a real-analytic 
(periodic) parametrization \(\gamma:\mathbb{T}\to \Lambda\). It follows that we can write 
$$
\int\limits_{\Lambda} g_\Omega(z,\zeta) d\lambda(\zeta) 
= \int\limits_{0}^{2\pi} g_\Omega(z,\gamma(e^{i\theta}))\varpi(\gamma(e^{i\theta}))
|\gamma'(e^{i\theta})| 
\frac{d\theta}{2\pi}.
$$
We now choose a particular parametrization \(\gamma\) satisfying 
the ODE \(\varpi(\gamma(e^{i\theta}))|\gamma'(e^{i\theta})|=1\).
This is possible to do with \(\gamma\) real-analytic, and corresponds 
to choosing an arc-length parametrization in the metric \(\varpi\,ds\).

We want to find points $\{\xi_{k,j}\}_{k=1}^{j}$ on $\mathbb{T}$ such that 
\begin{equation}\label{fourier_approx}
\int\limits_{0}^{2\pi} g_\Omega(z,\gamma(e^{i\theta})) 
\frac{d\theta}{2\pi} = \frac{1}{j}\sum\limits_{k=1}^{j} 
g_\Omega(z,\gamma(\xi_{k,j})) + \bigO(e^{-cj})
\end{equation}
as $j\to \infty$ for some constant $c>0$ uniformly in $z$. 
The lemma then follows by putting \(\zeta_{k,j}=\gamma(\xi_{k,j})\) for \(k=1,\ldots,j\). To this end, 
we write $g_\Omega(z,\gamma(e^{i\theta}))$ as a Fourier series
$$
g_\Omega(z,\gamma(e^{i\theta})) = \sum\limits_{n=-\infty}^{\infty} c_n(z) e^{i\theta n},
$$
and we get 
that the left-hand side of \eqref{fourier_approx} is equal to $c_0(z)$. 
We now pick the points $\{\xi_{k,j}\}$ as the $j$th roots of unity. We get
\begin{align*}
\frac{1}{j} \sum\limits_{k=1}^{j} g_\Omega(z,\gamma(\xi_{k,j}))  
= \frac{1}{j} \sum\limits_{k=1}^{j} \sum\limits_{n=-\infty}^{\infty} c_n(z) e^{2\pi i \frac{kn}{j}}. 
\end{align*}
Changing the order of summation and using that
$$
\sum\limits_{k=1}^j \frac{1}{j} e^{2\pi i \frac{kn}{j}} = \begin{cases}
1, \quad  j|n \\ 
0, \quad \text{otherwise},
\end{cases}
$$
we get
$$
\frac{1}{j} \sum\limits_{k=1}^{j} g_\Omega(z,\gamma(\xi_{k,j}))  
= c_0(z) + \sum\limits_{n\in \mathbb{Z}\setminus \{0\}} c_{nj}(z) .
$$
The function $g_\Omega(z,\gamma(e^{i\theta}))$ is real-analytic in a neighborhood of $\mathbb{T}$, 
hence its Fourier coefficient have exponential decay 
$c_n(z) = \bigO(e^{-cn})$ for some $c>0$. This completes the proof.
\end{proof}

We define holomorphic functions $e^{\mathcal{U}_{\lambda_j}(z)}$ and $e^{\mathcal{U}_{\lambda}(z)}$ on 
a neighborhood of $\Gamma$, in such a way that 
$$
e^{\re \mathcal{U}_{\lambda_j}(z)} = e^{U_\Omega^{\lambda_j}(z)}
$$ 
and 
$$
e^{\re \mathcal{U}_{\lambda}(z)} = e^{U_\Omega^{\lambda}(z)}.
$$
These functions are defined only up to an imaginary constant, 
and we choose this so that \(e^{\mathcal{U}_{\lambda_j}(z)}\) 
and \(e^{\mathcal{U}_{\lambda}(z)}\) are all real at the 
fixed point \(z_0\in\Gamma\). Since \(\lambda\) is a probability 
measure in \(\Omega\), an application of Green's formula shows that 
\(e^{\mathcal{U}_{\lambda}(z)}\) and \(e^{\mathcal{U}^{\lambda_j}}\) 
are single-valued on the neighborhood of \(\Gamma\). 

\begin{rem}
\label{rem:hol-U-lambda-j}
Note that the only singularities of \(e^{\mathcal{U}_{\lambda_j}(z)}\) 
are branch points of the form \((z-\zeta_{k,j})^{1/j}\) at each of the 
atoms of \(\lambda_j\), and hence \(e^{j \mathcal{U}_{\lambda_j}(z)}\) admits a holomorphic
extension to a neighborhood of the entire domain \(\Omega\) with simple zeros of \(\zeta_{k,j}\).
\end{rem}

From the error estimate in Lemma~\ref{point_approx} it follows that in a neighborhood of $\Gamma$ we have
\begin{equation}\label{eq:approx_point}
e^{-j\,\mathcal{U}_{\lambda}(z)} = \big(1+\bigO(e^{-cj})\big) e^{-j\,\mathcal{U}_{\lambda_j}(z)}.
\end{equation} 
We can now define the approximate reproducing kernel $K_n^{\sharp}(z,w)$ by
$$
K_n^{\sharp}(z,w) = e^{\frac{n}{2}(\mathcal{V}(z) +\overline{\mathcal{V}(w)} -Q(z)-Q(w) ) } 
\sum\limits_{j=m_n}^{M_n} h_{j,n}  e^{-j(\mathcal{U}_{\lambda_j}(z) + \overline{\mathcal{U}_{\lambda_j}(w)})},
$$
where $h_{j,n}$ are positive constants given by
$$
h_{j,n} = \frac{\rho(z_0)^2}{\nu(\Lambda)} \Big(n-\frac{j}{\nu(\Lambda)}\Big).
$$
We have some freedom in choosing $m_n$ and $M_n$. Let us say that
\begin{equation}
\label{eq:mn-Mn-hard-edge}
m_n = \lfloor n^{1/2} \rfloor, \qquad M_n = n\nu(\Lambda)(1-\epsilon_n),
\end{equation}
where $\epsilon_n$ is such that $M_n$ is an integer and satisfies
$$
\epsilon_n \to 0 , \qquad n\delta_n \epsilon_n \to +\infty,  
$$
as $n\to \infty$, where $\delta_n = \frac{\log n}{n}$. 

The approximate kernel is defined and holomorphic in a neighborhood of $\Omega$.

\begin{lem}\label{lem:local_rescale}  Locally 
uniformly for $\zeta,\eta\in \mathbb{H}$ we have
$$
\lim\limits_{n\to \infty} \frac{c(\zeta,\eta)}{(n\rho(z_0))^2}
\hspace{1pt} K_n^{\sharp}(z_n,w_n) =  B(\zeta,\eta).
$$
Furthermore, as $n\to\infty$ we have
\begin{equation}\label{norm_est_local_hard}
\frac{1}{(n \rho(z_0))^2}\int\limits_{\mathbb{D}(z_0,\delta_n)} 
|K_n^{\sharp}(z_n,w) |^2 1_{\C\setminus \Omega}(w)  dA(w) \leq  B(\zeta,\zeta) + o(1).
\end{equation} 
\end{lem}

\begin{proof} From equation \eqref{eq:approx_point} we see that the approximate kernel can be written as
$$
K_n^{\sharp}(z,w) =\big(1+ \bigO(e^{-cm_n})\big)\, 
e^{\frac{n}{2}(\mathcal{V}(z) +\overline{\mathcal{V}(w)} -Q(z)-Q(w))} 
\sum\limits_{j=m_n}^{M_n} h_{j,n} e^{-j(\mathcal{U}_{\lambda}(z)
+\overline{\mathcal{U}_{\lambda}(w)})}.
$$
We assume that $\zeta,\eta\in \mathbb{H}$. It follows from 
the proof of Theorem \ref{thm:limiting_hard} that
$$
 c(\zeta,\eta)  e^{\frac{n}{2}(\mathcal{V}(z_n) +\overline{\mathcal{V}(w_n)} -Q(z_n)-Q(w_n))}\  
= e^{-(\zeta + \bar{\eta})} \Big(1 + \bigO\Big(\frac{\|(\zeta,\eta)\|^2}{n^2}\Big)\Big).  
$$
From Lemma \ref{greens_potential} we get that
\begin{align}
\mathcal{U}_\lambda(z_n) &= \mathcal{U}_\lambda(z_0) + \mathcal{U}_\lambda'(z_0)(z_n-z_0)  
+ \bigO\Big(\frac{|\zeta|^2}{n^2}\Big) \\ &= \im \mathcal{U}_\lambda(z_0) - \frac{\zeta}{\nu(\Lambda)n} 
+ \bigO\Big(\frac{|\zeta|^2}{n^2}\Big),
\end{align}
and hence
$$
\mathcal{U}_\lambda(z_n) + \overline{\mathcal{U}_\lambda(w_n)} 
=- \frac{\zeta+\bar{\eta}}{\nu(\Lambda)n} + \bigO\Big(\frac{\|(\zeta,\eta)\|^2}{n^2}\Big).
$$
From this it follows that
\[
c(\zeta,\eta)e^{\frac{n}{2}(\mathcal{V}(z_n) +\overline{\mathcal{V}(w_n)} -Q(z_n)-Q(w_n))}
e^{-j(\mathcal{U}_{\lambda}(z_n)+\overline{\mathcal{U}_{\lambda}(w_n)})}=e^{(\zeta+\bar{\eta})
(\frac{j}{\nu(\Lambda)n}-1)}\Big(1+\bigO\Big(\frac{\|(\zeta,\eta)\|^2}{n}\Big)\Big),
\]
and hence we may write
\[
K_n^{\sharp}(z_n,w_n)c(\zeta,\eta)=
\sum\limits_{j=m_n}^{M_n} h_{j,n} e^{(\zeta+\bar{\eta})(\frac{j}{\nu(\Lambda)n}-1)}
+\bigO\big(n\|(\zeta,\eta)\|^2\big),
\]
where we have used the fact that \(|e^{(\zeta+\bar{\eta})(\frac{j}{\nu(\Lambda)n}-1)}|\le 1\) 
together with the crude bound \(h_{j,n}=\bigO(n)\) to bound the error term.
From the definition of $h_{j,n}$ we get
\begin{equation}
K_n^{\sharp}(z_n,w_n) c(\zeta,\eta) =\frac{n^2\rho(z_0)^2}{\nu(\Lambda)} 
\sum\limits_{j=m_n}^{M_n}\frac{1}{n} \Big(1-\frac{j}{\nu(\Lambda)n}\Big) 
e^{(\zeta+\bar{\eta})(\frac{j}{\nu(\Lambda)n}-1)} + \bigO\big(n\|(\zeta,\eta)\|^2\big).
\end{equation}

A Riemann sum approximation gives
\begin{align}
\sum\limits_{j=m_n}^{M_n}\frac{1}{n} \Big(1-\frac{j}{\nu(\Lambda)n}\Big) 
e^{(\zeta+\eta)(\frac{j}{\nu(\Lambda)n}-1)} &= \int\limits_{\frac{m_n}{n}}^{\frac{M_n}{n}} 
\Big(1-\frac{t}{\nu(\Lambda)}\Big) e^{(\zeta+\bar{\eta})(\frac{t}{\nu(\Lambda)}-1)} dt 
+ \bigO\Big(\frac{n\delta_n}{n}\Big)\\
& = \nu(\Lambda) \int\limits_{\frac{m_n}{\nu(\Lambda)n}}^{\frac{M_n}{\nu(\Lambda)n}} (t-1)
e^{(\zeta+\bar{\eta})(t-1)}dt +  \bigO\big(\delta_n\big)
\end{align}
uniformly for $\zeta,\eta\in \mathbb{H}$. It is now clear that locally uniformly in $\zeta,\eta$ we have
$$
\lim\limits_{n\to \infty} \frac{c(\zeta,\eta)}{(n\rho(z_0))^2}K_n^{\sharp}(\zeta,\eta)  = B(\zeta,\eta),
$$
since $\frac{m_n}{n\nu(\Lambda)} \to 0$ and $\frac{M_n}{n\nu(\Lambda)}\to 1$ as $n\to\infty$. 

To estimate the norm in \eqref{norm_est_local_hard} we will begin by showing that the contribution 
from $w =w(\eta) = z_0 + \frac{\eta}{n\rho(z_0)}$ such that $\eta\in \mathbb{L}$ is negligible. 
First, note that $|K_n^{\sharp}(z,w)|\leq Cn^{2}$ on 
$\C\setminus \Omega$. It will suffice to estimate the integral
$$
\int\limits_{\mathbb{D}(z_0,\delta_n)} 
1_{\C\setminus \Omega} (w) 1_{\mathbb{L}}(\eta) dA(w).
$$
This is the area between the tangent line and the boundary in the disc 
$\mathbb{D}(z_0,\delta_n)$ and is of order $\bigO(\delta_n^3)$. We see that
$$
\frac{1}{n^2 \rho(z_0)^2}\int\limits_{\mathbb{D}(z_0,\delta_n)} |K_n^{\sharp}(z,w)|^2 
1_{\C\setminus \Omega}(w) 1_{\mathbb{L}}(\eta)dA(w) = 
\bigO(n^2 \delta_n^3) = o(1),
$$ 
as $n\to\infty$. 

It remains to consider $\eta\in \mathbb{H}$ for which we have already derived an 
asymptotic expression. Let $f_{n,\zeta}(\eta)$ be the function
$$
f_{n,\zeta}(\eta) = 
\int\limits_{\frac{m_n}{\nu(\Lambda)n}}^{\frac{M_n}{\nu(\Lambda)n}}
(1-t) e^{(\zeta+\bar{\eta})(t-1)} dt.
$$
We have that
$$
\frac{1}{(n\rho(z_0))^2}K_n^{\sharp}(z_n,w)c(\zeta,\eta) 
= f_{n,\zeta}(\eta) + \bigO\Big(\frac{\log^2 n}{n}\Big),
$$
as $n\to \infty$, uniformly for $\eta = \bigO(n\delta_n)$. We obtain
\begin{align}
\frac{1}{(n\rho(z_0))^2} &\int\limits_{\mathbb{D}(z_0,\delta_n)} |K_n^{\sharp}(z_n,w)|^2 
1_{\C\setminus \Omega}(w) 1_{\mathbb{H}}(\eta) dA(w) \\
& \leq \frac{1}{(n\rho(z_0))^4}\int\limits_{\mathbb{D}(0,n\rho(z_0)\delta_n)} 
|K_n^{\sharp}(z_n,w(\eta))|^2 1_\mathbb{H}(\eta)dA(\eta) +o(1)\\
& \leq \int\limits_{\mathbb{H}} |f_{n,\zeta}(\eta)|^2 dA(\eta) +o(1 )\\
&\leq \int\limits_{\mathbb{H}} |B(\zeta,\eta)|^2 dA(\eta) +o(1)= B(\zeta,\zeta)+o(1).
\end{align}
The last inequality follows by Fatou's lemma since 
$f_{n,\zeta}(\eta)$ converges pointwise to $B(\zeta,\eta)$. 
\end{proof}

The following lemma proves that the norm of the approximate reproducing kernel 
$K_n^{\sharp}$ is small outside of the disc $\mathbb{D}(z_0,\delta_n)$. 
We define the smooth cut-off function $\chi: \C\to \R$ equal to 1 on a 
neighborhood of $\Omega$ and zero outside a slightly larger neighborhood.
The precise conditions on \(\chi\) will be clarified below, but here we only 
require that the approximate kernel \(K_n^\sharp(z,w)\) is well-defined 
wherever \(\chi\) does not vanish.

\begin{lem}\label{lem:off_diag_hard} Let $z_n = z_0 + \frac{e^{i\theta}\zeta}{\rho(z_0)n}$ 
with $\zeta$ in a compact set. We have 
$$
\frac{1}{(n\rho(z_0))^2}\int\limits_{\C\setminus \mathbb{D}(z_0,\delta_n)} 
|K_n^{\sharp}(z_n,w) \chi(w)|^2 1_{\C\setminus \Omega}(w) dA(w) = o(1),
$$
as $n\to \infty$. 
\end{lem}

\begin{proof} Because of the cut-off function $\chi$ we only have to consider 
a neighborhood of $\Gamma$ where we know that
$$
K_n^{\sharp}(z,w) =(1+ o(1))\hspace{1pt} e^{\frac{n}{2}(\mathcal{V}(z) +\overline{\mathcal{V}(w)})} 
e^{-\frac{n}{2}(Q(z)+Q(w))} \sum\limits_{j=m_n}^{M_n} h_{j,n} 
e^{-j(\mathcal{U}_{\lambda}(z)+\overline{\mathcal{U}_{\lambda}(w)})}.
$$

The sum has a simple structure and can be computed explicitly. For convenience we let 
$\Theta = e^{-(\mathcal{U}_{\lambda}(z)+\overline{\mathcal{U}_{\lambda}(w)})}$. 
Using that $|\Theta|\geq 1$ on $\C\setminus \Omega$ we get the bound 
\begin{equation}\label{geometric_bound}
\Big \lvert \sum\limits_{j=m_n}^{M_n} h_{j,n} \Theta^j \Big\rvert \lesssim 
n \frac{|\Theta|^{M_n}}{|\Theta-1|} + \frac{|\Theta|^{M_n}}{|\Theta-1|^2}.
\end{equation}
It remains to estimate the integral
\begin{multline}
\label{eq:exterior-norm-bd-approx-kernel-hard-edge}
\int\limits_{\C\setminus \mathbb{D}(z_0,\delta_n)}|K_n^{\sharp}(z,w)|^2 \chi(w) 
1_{\C\setminus \Omega}(w)\,dA(w) \\ 
\lesssim   \int\limits_{(\C\setminus \Omega)\setminus \mathbb{D}(z_0,\delta_n)} 
\Big(n \frac{|\Theta|^{M_n}}{|\Theta-1|} + \frac{|\Theta|^{M_n}}{|\Theta-1|^2}\Big)^2 
e^{n(V-Q)(z)+n(V-Q)(w)} dA(w).
\end{multline}
We claim that, as \(n\to\infty\),
\begin{equation}\label{eq:less_one}
|\Theta|^{2n\nu(\Lambda)} e^{n(V-Q)(z)} e^{n(V-Q)(w)} \leq 1+o(1),
\end{equation}
for $w$ in some one-sided neighborhood of $\Gamma$ in \(S\setminus\Omega\) 
and \(\zeta\) confined to a compact set. 
To see why, consider 
$$
G(w)=-2\nu(\Lambda)U_\Omega^{\lambda}(w) + V(w)-Q(w).
$$
Since $G(w) =0$ on $\Gamma$ and the normal derivative of $G$ is zero on 
$\Gamma$ we get $\nabla G(w) = 0$ for $z\in \Gamma$. This implies that, 
$\partial_{\mathbf{t}}^2 G(w)=0$ and we get
\begin{equation}\label{eq:second_normal}
\partial_{\mathbf{n}}^2 G(w) = 4\Delta G(w) = -4\Delta Q(w)<0.
\end{equation}
This shows that there is a one-sided neighborhood of $\Gamma$ in which \(e^{G(w)}\le 1\) holds. 
An inspection shows that
\[
|\Theta|^{2n\nu(\Lambda)} e^{n(V-Q)(z)} e^{n(V-Q)(w)}=e^{nG(z)+nG(w)}.
\]
For \(\zeta\) confined to a compact set, a Taylor 
expansion of \(G\) shows that \(nG(z_n)=\bigO(1/n)\)
and hence \eqref{eq:less_one} follows.

We now get some simple estimates for the integrals by integrating over the level 
curves $|\Theta| = v$. We make the change of variables $\Theta(w) = v(\theta) e^{2\theta i}$. 
The Jacobian of the change of variable can be bounded by a constant. We also have the estimate
$$
|\Theta-1| = | v(\theta) e^{2\theta i} -1 |\gtrsim |e^{2\theta i}-1| \gtrsim |\theta|.
$$

We now get the estimates
  
\begin{align}
\int\limits_{(\C\setminus \Omega)\setminus \mathbb{D}(z_0,\delta_n)} n^2 
\frac{|\Theta|^{2M_n-2n\nu(\Lambda)} }{|\Theta-1|^2} dA(w) &\lesssim n^2\int\limits_{\delta_n}^{1} \frac
{1}{|\theta|^2} d\theta \int\limits_{0}^1 v^{2M_n-2n\nu(\Lambda)} dv\\
&\lesssim n^2 \frac{1}{\delta_n} \frac{1}{\nu(\Lambda)n-M_n},
\end{align}
and
\begin{align}
\int\limits_{(\C\setminus \Omega)\setminus \mathbb{D}(z_0,\delta_n)} 
\frac{|\Theta|^{2M_n-2n\nu(\Lambda)} }{|\Theta-1|^4} dA(w) &\lesssim \int\limits_{\delta_n}^{1} 
\frac{1}{|\theta|^4} d\theta \int\limits_{0}^1 v^{2M_n-2n\nu(\Lambda)} dv\\
&\lesssim  \frac{1}{\delta_n^3} \frac{1}{\nu(\Lambda)n-M_n}.
\end{align}
Note that from the definition \eqref{eq:mn-Mn-hard-edge} of $M_n$ we have
$$
\delta_n (\nu(\Lambda)n- M_n) = \delta_n \epsilon_n n\nu(\Lambda) \to \infty,
$$   
as $n \to \infty$. This together with Cauchy-Schwarz' inequality 
shows that the right-hand side in \eqref{eq:exterior-norm-bd-approx-kernel-hard-edge} 
is of order \(o(n^2)\) as \(n\to\infty\), completing the proof. 
\end{proof}

\subsection{Correction to polynomial}
In order to use the reproducing property of $K_n$,  we will show that 
$K_n^{\sharp}(\cdot,w)$ can be corrected to a weighted polynomial in $\mathcal{W}_n$
at little cost. 

We will consider $w = w_n$, and try to find a function $u_w$ such that 
\begin{equation}\label{eq:correction_hard}
\chi(z) K_n^{\sharp}(z,w) e^{\frac{n}{2}Q(z)} - u_w(z)
\end{equation}
is a polynomial of degree at most $n-1$. The correction $u_w(z)$ should be small both in norm and pointwise.

If we can show that
\begin{equation}\label{eq:norm_hard_der}
\int\limits_{\C} \left| \bar{\partial}_z \Big( \chi(z) K_n^{\sharp}(z,w)e^{\frac{n}{2}Q(z)}\Big) 
\right|^2 e^{-nQ(z)} dA(z) = \bigO(e^{-cn}),
\end{equation}
for some $c>0$, a Hörmander estimate (see \cite{HedenmalmWennmanActa}*{Proposition~2.6}) 
will give the existence of $u_w$ such that
$$
\|u_w\|_{\mathcal{W}_n}^2 = \bigO(e^{-c'n}), 
$$
for some $c'>0$ and $u_w(z) = \bigO(z^{n-1})$ as $|z|\to \infty$. Since 
$\chi(z)K_n^{\sharp}(z,w)e^{\frac{n}{2}Q(z)}$  vanishes outside a fixed compact set,
it will follow that \eqref{eq:correction_hard} defines a polynomial of degree at most $n-1$.

The derivative in \eqref{eq:norm_hard_der} is non-zero only where $\bar\partial_z \chi(z) \neq 0$. 
It follows that it is enough to estimate the integral
$$
\int\limits_{N} |K_n^{\sharp}(z,w)|^2 dA(z),
$$
where $N = \{z : \bar\partial_z \chi(z) \neq 0\}$. From the definition of 
$K_n^{\sharp}(z,w)$ we get the upper bound 
$|K_n^{\sharp}(z,w) |^2 \leq K_n^{\sharp}(z,z) K_n^{\sharp}(w,w)$ 
and from local estimates from before $K_n^{\sharp}(w_n,w_n) = \bigO(n^2)$.

We consider the function 
$$
F_{j,n}(z) = e^{-j\mathcal{U}_{\lambda_j}(z)} e^{\frac{n}{2}\mathcal{V}(z)}.
$$
From the definition $F_{j,n}$ is analytic in a neighborhood of $\Omega$. 
Recall that $\chi$ is a smooth cut-off function equal to $1$ on a neighborhood 
of $\Omega$. We want to show that $F_{j,n}(z)e^{-\frac{n}{2}Q(z)}$ is small 
on the support of $\bar{\partial}\chi$. 
On a neighborhood of $\Gamma$ we have
\begin{align*}
|F_{j,n}(z)|^2 e^{-nQ(z)} = (1+\bigO(e^{-cj})) e^{-2jU_\Omega^{\lambda}(z)} e^{nV(z)} e^{-nQ(z)}. 
\end{align*}
We want to say that the function
$$
-2jU_\Omega^{\lambda}(z) + nV(z) -nQ(z),
$$
is large and negative outside a small neighborhood of $\Omega$. 
To begin with note that the function $(\nu(\Lambda)n-j)U_\Omega^{\lambda}$ 
is zero on $\Gamma$ and has negative normal derivative. 
It suffices to study the function 
$$
G(z)=-2\nu(\Lambda)U_\Omega^{\lambda}(z) + V(z)-Q(z).
$$
We have shown (see \eqref{eq:second_normal}) that 
$$
\partial^2_{\mathbf{n}} G(z) < 0,
$$
and since it vanishes along the boundary together with its first 
normal derivative it follows that, on the set \(N\),
$$
|F_{j,n}(z)|^2 e^{-nQ(z)} = \bigO(e^{-cn}),
$$
for some constant $c>0$.  From the definition of the approximate 
kernel it is clear that the same estimate applies for $K_n^{\sharp}(z,w)$.

The function $u_w$ can also be estimated pointwise. On the set where $\chi(z)=1$ 
we can use a subharmonicity estimate cf.\ \cite{RescalingWardAKM}*{Lemma 3.3} and obtain that
$$
u_w(z)e^{-\frac{n}{2}Q(z)} = \bigO(e^{-cn}),
$$
with possibly a smaller constant $c>0$.

\subsection{Proof of Theorem \ref{thm:main-hard-edge}}
\noindent Let $n_k$ be a sequence along which we have convergence of the rescaled 
kernel $c(\zeta,\eta)\cdot k_n(\zeta,\eta)$ locally uniformly to a limiting kernel 
$K(\zeta,\eta)$.  From Corollary \ref{cor:upper_bound_hard} we have the upper bound 
$K(\zeta,\zeta)\leq B(\zeta,\zeta)$. 

We will now prove the opposite inequality, using the approximate reproducing kernel 
$K_n^\sharp$. We have since $K_n$ is the reproducing kernel of the space $\mathcal{W}_n$ that
$$
K_n(z,z) = \sup\limits_{f\in\mathcal{W}_n\setminus \{0\}} \frac{|f(z)|^2 }{\| f\|_{\mathcal{W}_n}^2}.
$$ 
We have proven that $\chi\, K_n^{\sharp}(\cdot,z_n)-u_{z_n}e^{-\frac{n}2 Q}$ is 
a weighted polynomial of degree at most $n-1$. 
Since $u_{z_n}e^{-\frac{n}2 Q}$ is small both pointwise and in norm, we get
\begin{equation}\label{eq_pol_rep}
K_n(z_n,z_n) \geq (1+o(1))\frac{K_n^{\sharp}(z_n,z_n)^2}{\|K_n^{\sharp}(\cdot,z_n)
\chi(\cdot)\|_{\mathcal{W}_n}^2}.
\end{equation}
From Theorem \ref{thm:limiting_hard} and Lemma \ref{lem:local_rescale} we have that
$$
\lim\limits_{k\to\infty} \frac{1}{(n_k\rho(z_0))^2} K_{n_k}(z_{n_k},z_{n_k}) = K(\zeta,\zeta)
$$
and
$$
\lim\limits_{k\to\infty} \frac{1}{(n_k\rho(z_0))^2} K_{n_k}^{\sharp}(z_{n_k},z_{n_k}) 
= B(\zeta,\zeta).
$$
Lemma \ref{lem:local_rescale} and Lemma \ref{lem:off_diag_hard} gives that
$$
\lim\limits_{k\to \infty} \frac{1}{(n_k\rho(z_0))^2}\|K_{n_k}^{\sharp}(\cdot,z_{n_k})
\chi(\cdot)\|_{\mathcal{W}_n}^2 \leq B(\zeta,\zeta).
$$
Combining these three results with \eqref{eq_pol_rep} we obtain the lower bound
$$
K(\zeta,\zeta)\geq B(\zeta,\zeta).
$$
We have already proved the reversed inequality, which means that $K(\zeta,\zeta) =B(\zeta,\zeta)$. 
Both functions are Hermitian-analytic so they are equal not only on the diagonal, but everywhere. 
This proves that the limit point in Theorem \ref{thm:limiting_hard} is unique, which completes the proof.
\hfill \(\square\)

\section{Soft and Soft/hard edges}
In this section we will prove Theorems~\ref{thm:main-soft-edge} and 
Theorem~\ref{thm:main-soft-hard-edge}. We recall that $Q$ is a real-analytic 
potential with droplet $S$ subject to the standing assumptions from Section~\ref{s:main-results}.
In particular, the boundary \(\partial S\) consists of finitely many real-analytic Jordan curves. 

The following lemma is well known, and plays the same role as as 
Lemma~\ref{lem:apriori_hard} did in the hard-edge case.

\begin{lem}\label{lem:apriori_soft} The kernel $K_n(z,w)$ satisfies
$$
K_n(z,z) \leq C n e^{n(\check{Q}(z)-Q(z))}, \quad z\in \C.
$$ 
\end{lem}

\subsection{Limiting kernel at soft edges}

We define the rescaled kernel 
\[
k_n(\zeta,\eta) = \frac{1}{n\Delta Q(z_0)}K_n(z_n,w_n)
\]
where \(z_n\) and \(w_n\) denote the rescaled variables 
\begin{equation}
\label{eq:rescaling-soft}
z_n = z_0+\frac{e^{i\theta}\zeta}{\sqrt{n\Delta Q(z_0)}} \quad \text{ and } \quad
w_n =z_0 +\frac{e^{i\theta}\eta}{\sqrt{n\Delta Q(z_0)}} ,
\end{equation}
respectively. Here $e^{i\theta}$ is the outer unit normal of $S$ at $z_0$.

The following theorem is proven in \cite{RescalingWardAKM}*{Theorem 3.5}.
\begin{thm}\label{thm:limiting_soft} 
There exist cocycles $c_n(\zeta,\eta)$ such that $(c_n\cdot k_n)_{n=0}^{\infty}$ is a normal family. 
Any limit point takes the form \(K(\zeta,\eta)={\bf K}(\zeta,\eta) e^{-\frac12(|\zeta|^2+|\eta|^2)}\), 
where \({\bf K}(\zeta,\eta)\) is Hermitian-entire and satisfies the \enquote{mass-one inequality}
$$
\int\limits_{\C} |\mathbf{K}(\zeta,\eta)|^2 e^{-|\eta|^2} dA(\eta) \leq \mathbf{K}(\zeta,\zeta) 
$$ 
as well as the exterior decay estimate
$$
|\mathbf{K}_\eta(\zeta) e^{\zeta^2/2} | \leq C_\eta, \quad \re \zeta \geq 0.
$$
\end{thm}

Theorem \ref{thm:limiting_soft} gives an embedding theorem in a 
similar way to what we saw in the Hard edge case.
Let $F^2(\C)$ be the Fock space of entire functions with finite norm
$$
\| f\|_{F^2(\C)}^2 = \int\limits_{\C} |f(\zeta)|^2  e^{-|\zeta|^2} dA(\zeta) < +\infty.
$$
The Fock space $F^2(\C)$ is isometrically isomorphic to $L^2(\R)$ 
via the Bargmann transform $\mathcal{B}:L^2(\R)\to F^2(\C)$ given by
$$
\mathcal{B} [F](\zeta) = \frac{1}{(2\pi)^{1/4}}\int\limits_{\R} e^{\zeta t} 
e^{-\zeta^2/2} e^{-t^2/4} F(t) dt. 
$$

We consider the space $\mathcal{H}=\mathcal{H}_\C$ of 
entire functions with finite Fock space norm
$$
\|f\|_{F^2(\C)}^2 = \int\limits_{\C} |f(\zeta)|^2 e^{-|\zeta|^2} dA(\zeta) <+\infty
$$ 
and with the growth restriction
$$
|f(\zeta) e^{\zeta^2/2}| \leq C, \quad \re \zeta \geq 0.
$$

We recall the definition of the entire function
\[
\Phi(z)=\frac12\operatorname{erfc}\left(\frac{z}{\sqrt{2}}\right)
=\frac{1}{\sqrt{2\pi}}\int_z^\infty e^{-t^2/2}dt.
\]

\begin{thm} 
\label{thm:erfc-space} 
We have that \(\mathcal{H}_\C=\mathcal{B}[L^2(-\infty,0)]\) and 
hence Hermitian-entire function 
$H(\zeta,\eta)=e^{\zeta \bar{\eta}} \Phi(\zeta +\bar{\eta})$ 
is the reproducing kernel of the space $\mathcal{H}_\C$.
\end{thm}
\begin{proof} We have already that $\mathcal{H}_\C$ is a subspace of $F^2(\C)$. 
Hence any $f\in\mathcal{H}_\C$ is given by $f= \mathcal{B}[F]$ with $F\in L^2(\R_-)$. 
From the Paley-Wiener-Schwartz theorem and the growth constraint on functions in 
$\mathcal{H}_\C$ it follows that $f$ is the Bargmann transform of a distribution 
\(u\) supported in $\mathcal{\R}_-$. The same argument as in the proof of 
Theorem~\ref{thm:space-A} then shows that \(u=F\), so that in fact 
$\mathcal{H}_\C=\mathcal{B}[L^2(\R_-)]$. It follows from 
\cite{HaimiHedenmalm}*{Proposition~5.4} that 
$e^{\zeta\bar{\eta}} \Phi(\zeta +\bar{\eta})$ is the reproducing kernel 
of $\mathcal{B}[L^2(\R_-)]$ and so also of $\mathcal{H}_\C$. 
\end{proof}

By combining Theorems~\ref{thm:limiting_soft} and \ref{thm:erfc-space} we obtain 
the following corollary. The proof is directly parallel with the proof of 
Corollary~\ref{cor:upper_bound_hard} in the hard-edge case
and we omit the details.

\begin{cor}\label{cor: upper_bound_soft}
Any limiting kernel $\mathbf{K}_\zeta$ from Theorem \ref{thm:limiting_soft} 
belongs to the space $\mathcal{H}_\C$ and satisfies 
$\mathbf{K}(\zeta,\zeta)\leq {\bf H}(\zeta,\zeta)$. 
\end{cor}

\subsection{Limiting kernel soft/hard edge}

In \cite{AmeurHardEdge}*{Lemma~2.4} it was proven that Lemma \ref{lem:apriori_soft} 
also holds in the soft/hard-edge setting. With some minor modifications, the proof 
of Theorem \ref{thm:limiting_soft} then gives the following result.
We omit the proof.
 
\begin{thm}\label{thm:limiting_softhard} There exist cocycles $c_n(\zeta,\eta)$ 
such that $(c_n\cdot k_n)_{n=0}^{\infty}$ is a normal family. Any limit point 
takes the form \(K(\zeta,\eta)={\bf K}(\zeta,\eta)e^{-\frac12(|\zeta|^2+|\eta|^2)}\), 
where the Hermitian-entire function ${\bf K}(\zeta,\eta)$ satisfies the 
\enquote{mass-one inequality}
$$
\int\limits_{\mathbb{L}} |\mathbf{K}(\zeta,\eta)|^2 
e^{-|\eta|^2} dA(\eta) \leq \mathbf{K}(\zeta,\zeta)
$$
and the exterior decay estimate
$$
|\mathbf{K}_\eta(\zeta) e^{\zeta^2/2} | \leq C_\eta, \quad \re \zeta \geq 0.
$$
\end{thm}

We denote by $F^2(\mathbb{L})$ the space of holomorphic functions on 
the left half-plane with finite norm
$$
\int\limits_{\mathbb{L}} |f(\zeta)|^2 e^{-|\zeta|^2} dA(\zeta)<+\infty,
$$
and define the space $\mathcal{H}_\mathbb{L}$ as the subspace 
of $F^2(\mathbb{L})$ of entire functions satisfying
$$
|f(\zeta)e^{\zeta^2/2}|\leq C , \quad \re \zeta \geq 0.
$$
We have the following characterization of \(\mathcal{H}_\mathbb{L}\).

\begin{thm}
\label{thm:hard-edge-erfc-space} The space $\mathcal{H}_\mathbb{L}$ 
coincides with $\mathcal{H}_\C$ as sets. The reproducing kernel of 
$\mathcal{H}_\mathbb{L}$ is given by 
${\bf H}_{\mathbb{L}}(\zeta,\eta)=e^{\zeta\bar{\eta}} \Psi(\zeta+\bar{\eta})$, 
where \(\Psi\) is given by \eqref{eq:soft/hard-bdry-fcn}. 
\end{thm}

\begin{proof}
We begin by showing that $\mathcal{H}_\mathbb{L}\subset F^2(\C)$. 
Let $f$ be a function in $\mathcal{H}$. From the Paley-Wiener-Schwartz theorem it 
follows that $f$ is the Bargmann transform of a distribution $u$ supported on $\R_-$. 
It only remains to show that $u\in L^2(\R)$. Since $f\in F^2(\mathbb{L})$ we get from 
Fubini's theorem that $f(x+iy)e^{(x+iy)^2/2}$ is in $L^2(\R)$ as a function of $y$ 
for a.e.\ $x<0$. For such an $x$, we have that
$$
f(x+iy)e^{(x+iy)^2/2} = \int\limits_{-\infty}^0  e^{xt} e^{iyt}e^{-t^2/4} u(t) dt,
$$
with $e^{xt} e^{-t^2/4}u(t)\in L^2(\R)$ by Plancherel's theorem. From the support 
of $u$ we see that $e^{xt}$ is decreasing in $x$. This gives the inequality
\begin{align}
N(x):=\int_{-\infty}^\infty |f(x+iy)|^2e^{-x^2-y^2}dy 
&= e^{-2x^2}\int_{-\infty}^\infty |f(x+iy) e^{(x+iy)^2/2}|^2 dy\\
&=2\pi e^{-2x^2}\int_{-\infty}^0 e^{2xt} e^{-t^2/2}e^{-t^2/4} |u(t)|^2dt.
\end{align}
It follows \(N(x)<\infty\) for a.e.\ \(x\in \R\), and that \(N(|x|)\le N(-|x|)\).
In particular we obtain
$$
\|f\|_{F^2(\mathbb{C})}^2 =\frac{1}{\pi}\int\limits_{\R}N(x)dx 
\leq \frac{2}{\pi}\int\limits_{\R_{-}}N(x)dx=2 \|f\|_{F^2(\mathbb{L})}^2 <\infty.
$$
Since $f\in F^2(\C)$ we have $f = \mathcal{B}[F]$ for $F$ in $L^2(\R)$ 
and in the same way as above, we get that $u=F$. 

Next we show that $\mathcal{B}$ is an isometry
$$
\mathcal{B}:L^2(\R_-,\Phi(t) dt) \to {F}^2(\mathbb{L}).
$$ 
We compute the norm of $f=\mathcal{B}[F]$ where $F\in L^2(\R_-, \Phi(t)dt)$,
\begin{align*}
\int\limits_{\mathbb{L}} |f(\zeta)|^2 e^{-|\zeta|^2} dA(z) 
&= \int\limits_{-\infty}^{0} \int\limits_{-\infty}^{\infty} 
|\mathbf{B}[F](x+iy)|^2 e^{-x^2-y^2} \frac{dydx}{\pi}  \\
& = \int\limits_{-\infty}^{0} \int\limits_{-\infty}^{\infty} 
\Big\lvert (2\pi)^{-1/4}\int\limits_{-\infty}^{0} F(t) e^{zt} 
e^{-z^2/2} e^{-t^2/4} dt\Big\rvert^2 e^{-x^2-y^2}\frac{dydx}{\pi}\\
& =  \int\limits_{-\infty}^{0} \int\limits_{-\infty}^{\infty} \Big\lvert  
\int\limits_{-\infty}^{0} F(t) e^{zt}  e^{-t^2/4} dt\Big\lvert^2 e^{-2x^2}
\frac{dydx}{\sqrt{2\pi}\pi}.
\end{align*}
Using that $\int\limits_{-\infty}^{0} F(t) e^{zt} e^{-t^2/4} dt$ is the Fourier 
transform of the function $1_{\R_-}(t)F(t)e^{xt} e^{-t^2/4}$ and Plancherel's theorem we get
\begin{align*}
\int\limits_{\mathbb{L}} |f(\zeta)|^2 e^{-|\zeta|^2} dA(z) 
&= \int\limits_{-\infty}^{0} \int\limits_{-\infty}^{\infty} 
1_{\R_{-}}(y)|F(y)e^{xy}e^{-y^2/4}|^2 e^{-2x^2} 2\pi\frac{dydx}{\sqrt{2\pi}\pi} \\
&= \int\limits_{-\infty}^{0} |F(y)|^2 \int\limits_{-\infty}^{0} 
e^{-2x^2} e^{2xy} e^{-y^2/2} 2\pi\frac{dxdy}{\sqrt{2\pi}\pi}\\
&= \int\limits_{-\infty}^{0} |F(y)|^2 \Phi(y) dy,
\end{align*}
where the last step follows from a change of variables.

Let $L(\zeta,\eta) = L_\eta(\zeta)$ be the reproducing 
kernel of $\mathcal{B}[L^2(\R_-)]$ as a subset of ${F}^2(\mathbb{L})$ 
with norm and inner product from $L^2(\mathbb{L},e^{-|\zeta|^2}dA)$.
We get that 
$$
\langle \mathcal{B}[F],L_\eta \rangle_{{F}^2(\mathbb{L})} 
= \langle F,\mathcal{B}^{-1}[L_\eta] \rangle_{L^2(\R,\Phi(t) dt)},
$$  
for any function $F\in L^2(\R_-,\Phi(t)dt)$. 
From the definition of the Bargmann transform,
$$
\mathcal{B}^{-1}[L_\eta](t) = \frac{1}{(2\pi)^{1/4}} 1_{\R_-}(t) 
e^{t\bar{\eta}-\bar{\eta}^2/2-t^2/4} \frac{1}{\Phi(t)}. 
$$
If we apply the Bargmann transform to both sides we obtain 
\begin{align*}
L_\eta(\zeta) &= \frac{1}{\sqrt{2\pi}} \int\limits_{-\infty}^{0} e^{t(\zeta + \bar{\eta)}} 
e^{-\zeta^2/2 -\bar{\eta}^2/2 -t^2/2} \frac{1}{\Phi(t)} dt\\
&=e^{\zeta \bar{\eta}} \frac{1}{\sqrt{2\pi}} \int\limits_{-\infty}^{0} 
e^{-(t-\zeta -\bar{\eta})^2/2} \frac{1}{\Phi(t)} dt \\
&= e^{\zeta \bar{\eta}} \Psi(\zeta+\bar{\eta})={\bf H}_{\mathbb{L}}(\zeta,\eta),
\end{align*}
which is what we wanted to prove.   
\end{proof}

Similarly as in the soft-edge case, the following corollary follows by combining 
Theorems~\ref{thm:limiting_softhard} and \ref{thm:hard-edge-erfc-space}.
\begin{cor}\label{cor:upper_bound_soft_hard}
Any limiting kernel $\mathbf{K}_\eta$ from Theorem \ref{thm:limiting_softhard} 
belongs to the space $\mathcal{H}_\mathbb{L}$ and satisfies 
$\mathbf{K}(\zeta,\zeta)\leq {\bf H}_{\mathbb{L}}(\zeta,\zeta)$.
\end{cor}

\subsection{Approximate kernel and the lower bound}
\label{s:soft-edge-constr}
The construction of the approximate kernel will be almost the same in the soft 
and soft/hard edge cases and we will treat both cases simultaneously. Let $U$ 
be the unbounded component of $\C\setminus S$. In the proof we will assume that 
$z_0\in \partial U$. The other cases can be treated in more or less the same way. 

Let $\Gamma$ denote the boundary component of $U$ to which $z_0$ belongs. 
By assumption $\Gamma$ is real-analytic. We let $\mathcal{D}$ denote the 
unbounded component of $\C\setminus \Gamma$.

Let $\check{Q}$ be the solution to the obstacle problem on $\C$, that is the largest 
subharmonic function bounded above by $Q$ on $\C$ with growth at most 
$\log |z|^2 +\bigO(1)$ as $|z|\to \infty$. It is well known that the function 
$\check{Q}$ is harmonic in $U$.  We let $V$ be the harmonic continuation of 
$\check{Q}$ across $\partial U$ and define the holomorphic function 
$e^{\mathcal{V}}$ whose modulus equals $e^{V}$ and which is real-valued 
at \(z_0\). We can think of the function $\mathcal{V}$ as a multi-valued 
function defined in a neighborhood of $U$ with single-valued real part.

There exists however harmonic functions $\omega_n(z)$ in $U$ which vanish along 
$\Gamma$ and are equal to constants $\lambda_{j,n}\in [0,1]$ on all the other 
boundary components of $U$, such that the function 
$\Omega_n(z)=\omega_n(z) + i \widetilde{\omega}_n(z)$ has the opposite 
multi-valuedness of $\frac{n}{2}\mathcal{V}(z)$. That is, 
$e^{\Omega_n(z)+ \frac{n}{2}\mathcal{V}(z)}$ can be realized as a 
single-valued holomorphic function in a neighborhood of $U$. The 
construction of such functions is standard and we refer to \cite{bell2015cauchy} 
for the details.

In preparation for the construction of the approximate kernels, we want to 
construct a subharmonic function which is zero on the boundary $\Gamma$ and 
has normal derivative $\sqrt{\Delta Q}$ on $\Gamma$. We mimic the construction 
in the hard-edge setting and introduce a measure $\mu = \sqrt{\Delta Q} \,ds$, 
where $ds$ is arc-length measure on $\Gamma$ normalized by \(\frac1{2\pi}\). 

The following two lemmas are proven in the same way as 
Lemmas~\ref{lem:balayage-hard-edge} and \ref{greens_potential}, 
and we omit the details. 

\begin{lem} There exists a real-analytic Jordan curve $\Lambda$ in $\mathcal{D}$ 
and a measure $\nu$ with smooth density on $\Gamma$ such that $\widehat{\nu} = \mu$, 
where $\widehat{\nu}$ denotes the balayage of $\nu$ onto $\C\setminus \mathcal{D}$. 
\end{lem} 

\begin{lem}
\label{lem:greeb-pot-soft}
The Green's potential $U_\mathcal{D}^{\nu}$  is zero on $\Gamma$ and 
$\partial_{\mathbf{n}_i} U_\mathcal{D}^{\nu} = \sqrt{\Delta Q(z_0)}$. 
\end{lem}

In the same way as in the hard-edge case we define the probability 
measure $\lambda = \frac{\nu}{\nu(\Lambda)}$. We can find atomic probability 
measures $\lambda_j$ with atoms of equal size such that
$$
e^{jU_\mathcal{D}^\lambda(z)} = e^{jU_\mathcal{D}^{\lambda_j}(z)} (1 + \bigO(e^{-cj})).
$$ 

As in Section~\ref{s:hard-edge-constr}, we define holomorphic functions 
$e^{\mathcal{U}_{\lambda_j}}$ and \(e^{\mathcal{U}_{\lambda}(z)}\) on a 
neighborhood of \(\Gamma\) such that
$$
e^{\re \mathcal{U}_{\lambda_j}(z)} = e^{U_\mathcal{D}^{\lambda_j}(z)},
$$
and 
$$
e^{\re \mathcal{U}_{\lambda}(z)} = e^{U_\mathcal{D}^{\lambda}(z)}.
$$
By Remark~\ref{rem:hol-U-lambda-j}, the function \(e^{-j\mathcal{U}_{\lambda_j}(z)}\) 
is holomorphic on a full neighborhood of \(\mathcal{D}\), and we moreover have that
$$
e^{-j\mathcal{U}_{\lambda}(z)}  = (1 + \bigO(e^{-cj})) e^{-j\mathcal{U}_{\lambda_j}(z)}
$$
on a neighborhood of \(\Gamma\).

We now have to distinguish between the soft and soft/hard edge. 
Define the approximate kernel $K_n^{\sharp}(z,w) $ by
\begin{equation}
\label{eq:K-n-def-soft}
K_n^{\sharp}(z,w) = e^{\Omega_n(z)+\overline{\Omega_n(w)}} 
e^{\frac{n}{2}(\mathcal{V}(z) + \overline{\mathcal{V}(w)})} e^{-\frac{n}{2}(Q(z)+Q(w))} 
\sum\limits_{j=m_n}^{M_n}  h_{j,n} e^{-j(\mathcal{U}_{\lambda_j}(z)
+ \overline{\mathcal{U}_{\lambda_j}(w)})} .
\end{equation}
with $h_{j,n}$ given by
\begin{equation}\label{eq:h_soft}
h_{j,n} =  \sqrt{\frac{n}{2\pi}} u(j/\sqrt{n})
\frac{\Delta Q(z_0)}{\nu(\Lambda)}e^{-\frac{j^2}{2n\nu(\Lambda)^2}},
\end{equation}
where 
\[
\begin{cases}
u(t) =1&\text{(soft-edge case)}, \\
u(t)=\Phi(-t/\nu(\Lambda))^{-1}& \text{(soft/hard edge case)}.
\end{cases}	
\]
As in the hard-edge case we have some freedom when choosing 
$m_n$ and $M_n$. Let $m_n = \lfloor n^{1/4}\rfloor$. 
We let \(\delta_n=n^{-1/2}\log n\) and choose $M_n$ as 
sequence of integers such that 
\begin{equation}
\label{eq:def-mn-Mn-soft}
\frac{M_n}{\sqrt{n}} \to \infty , \qquad \frac{M_n}{\delta_n n}\to 0, 
\end{equation}
as $n\to\infty$. 

\begin{lem}\label{lem:diagonal_approx}  
In the soft-edge case, we have locally uniformly that 
$$
\lim\limits_{n\to \infty} \frac{c_n(\zeta,\eta) }{n\Delta Q(z_0)} 
K_n^{\sharp}(z_n,w_n) =  H(\zeta,\eta),
$$
where \(z_n\) and \(w_n\) are given by \eqref{eq:rescaling-soft}.
Furthermore, for $\zeta$ fixed we have
$$
\frac{1}{n\Delta Q(z_0)}\int\limits_{\mathbb{D}(z_0,\delta_n)} 
|K_n^{\sharp}(z_n,w)|^2 dA(w) \leq H(\zeta,\zeta)+o(1), 
$$
as $n\to\infty$.
\end{lem}

\begin{proof} 
First we get that
$$
\Omega_n(z_n) + \overline{\Omega_n(w_n)} = \bigO\Big(\frac{\|(\zeta,\eta)\|}{\sqrt{n}}\Big).
$$

Next, since $\nabla (V-Q) = 0$ on $\Gamma$, Taylor's formula gives the expansion
$$
\frac{n}{2}(\mathcal{V}(z_n)+ \overline{\mathcal{V}(w_n)}-Q(z_n) -Q(w_n)) 
= -(\re \zeta)^2 -(\re \eta)^2 + ig_n(\zeta,\eta) +\bigO\Big(\frac{\|(\zeta,\eta)\|^3}{\sqrt{n}}\Big),
$$
where $g_n(\zeta,\eta)$ is such that $e^{ig_n}$ is a cocycle. Another 
application of Taylor's formula to $\mathcal{U}_{\lambda}$ around the point $z_0$ gives
\begin{equation}
\label{eq:U-lambda-Taylor}
\mathcal{U}_\lambda(z_n) + \overline{\mathcal{U}_\lambda(w_n)} 
=  \frac{\zeta + \bar{\eta}}{\nu(\Lambda)\sqrt{n}} + \bigO\Big(\frac{\|(\zeta,\eta)\|^2}{n}\Big).
\end{equation}
Let $w\in \mathbb{D}(z_0,\delta_n)$ and write $w=z_0 + \frac{e^{i\theta}\eta}{n\rho(z_0)}$. 
Using that \(|\zeta|,|\eta|=\bigO(\log n)\) we get that
\begin{align*}
K_n^{\sharp}(z_n,w) & = e^{ig_n(\zeta,\eta)}\sum\limits_{j=m_n}^{M_n} h_{j,n} 
e^{-(\re \zeta)^2 -(\re \eta)^2} e^{-\frac{j}{\nu(\Lambda)\sqrt{n}}(\zeta +\bar{\eta})}
\Big(1+\bigO\Big(\frac{\log^3 n}{\sqrt{n}}\Big)\Big) \\ 
&= e^{ig_n(\zeta,\eta)}\sum\limits_{j=m_n}^{M_n} h_{j,n} e^{-(\re \zeta)^2 -(\re \eta)^2} 
e^{-\frac{j}{\nu(\Lambda)\sqrt{n}}(\zeta +\bar{\eta})}
+\bigO\big(\sqrt{n}\log^3 n\big).
\end{align*}
A Riemann sum approximation using the definition \eqref{eq:h_soft} of \(h_{j,n}\) gives that
\begin{align*}
&\frac{1}{n\Delta Q(z_0)}\sum\limits_{j=m_n}^{M_n} h_{j,n} 
e^{-(\re \zeta)^2 -(\re \eta)^2} e^{-\frac{j}{\nu(\Lambda)\sqrt{n}}(\zeta +\bar{\eta})}\\
&=  \frac{1}{\sqrt{2\pi}}\frac{1}{\nu(\Lambda)}\int\limits_{\frac{m_n}{\sqrt{n}}}^{\frac{M_n}{\sqrt{n}}} 
e^{-(\re \zeta)^2 -(\re \eta)^2} e^{-t\frac{(\zeta +\bar{\eta})}{\nu(\Lambda)}} 
e^{- \frac{t^2}{2\nu(\Lambda)^2}} dt + \bigO(\delta_n)\\
&= \frac{1}{\sqrt{2\pi}}\int\limits_{\frac{m_n}{\nu(\Lambda)\sqrt{n}}}^{\frac{M_n}{\nu(\Lambda)\sqrt{n}}} 
e^{-(\re \zeta)^2 -(\re \eta)^2} e^{-t(\zeta +\bar{\eta})} e^{- \frac{t^2}{2}} dt + \bigO(\delta_n),
\end{align*}
where we have used that the derivative of
$$e^{-(\re \zeta)^2 -(\re \eta)^2} e^{-t(\zeta +\bar{\eta})} e^{- t^2/2}$$
is uniformly bounded by $\bigO(\eta)$ for 
$t\in I_n =[\frac{m_n}{\nu(\Lambda)\sqrt{n}},\frac{M_n}{\nu(\Lambda)\sqrt{n}}]$.
Rewriting the last integral gives that
$$
\frac{1}{n\Delta Q(z_0)} K_n(z_n,w) c_n(\zeta,\eta) 
= f_{n,\zeta}(\eta) + \bigO\Big(\frac{\log^3 n}{\sqrt{n}}\Big), 
$$
where the function $f_{n,\zeta}$ be given by 
$$
f_{n,\zeta} (\eta) = G(\zeta,\eta) \frac{1}{\sqrt{2\pi}}
\int\limits_{\frac{m_n}{\nu(\Lambda)\sqrt{n}}}^{\frac{M_n}{\nu(\Lambda)\sqrt{n}}} 
e^{-(\zeta + \bar{\eta} +t)^2/2} dt.
$$    

The first claim of the theorem is now obvious since $\frac{M_n}{\sqrt{n}} \to \infty$ 
and $\frac{m_n}{\sqrt{n}} \to 0$ as $n\to \infty$. With the notation 
\(w(\eta)=z_0+e^{i\theta}\eta/\sqrt{n\Delta Q(z_0)}\), the second claim follows from
\begin{align}
\frac{1}{n\Delta Q(z_0)}\int\limits_{\mathbb{D}(z_0,\delta_n)} |K_n^{\sharp}(z_n,w)|^2 dA(w) 
&=  \frac{1}{(n\Delta Q(z_0))^2} \int\limits_{\mathbb{D}(z_0,\sqrt{\Delta Q(z_0)}\log n)} 
|K_n^{\sharp}(z_n,w(\eta))|^2 dA(\eta)\\
& = \int\limits_{\mathbb{D}(z_0,\sqrt{\Delta Q(z_0)}\log n)} |f_\zeta(\eta)|^2 dA(\eta) + o(1)\\
&\leq \int\limits_{\C} |H(\zeta,\eta)|^2 dA(\eta) + o(1)= H(\zeta,\zeta) + o(1),
\end{align}
as $n\to \infty$, where the asymptotic inequality follows from Fatou's lemma. 
\end{proof}

\begin{lem}\label{lem:diagonal_approx_softhard}  
In the soft/hard-edge case, we have locally uniformly in $\mathbb{H}$ that
$$
\lim\limits_{n\to \infty} \frac{c_n(\zeta,\eta) }{n\Delta Q(z_0)}K_n^{\sharp}(z_n,w_n) 
= H_{\mathbb{L}}(\zeta,\eta)
$$
where \(z_n\) and \(w_n\) are given by \eqref{eq:rescaling-soft}.
Furthermore, for $\zeta$ fixed we have  that
$$
\frac{1}{n\Delta Q(z_0)}\int\limits_{\mathbb{D}(z_0,\delta_n)} 
|K_n^{\sharp}(z_n,w)|^2 1_{S}(w) dA(w) \leq H_{\mathbb{L}}(\zeta,\zeta) + o(1),
$$
as $n\to\infty$. 
\end{lem}

\begin{proof}
Repeating the argument from Lemma~\ref{lem:diagonal_approx} verbatim, 
but using the soft/hard-edge definition of \(h_{j,n}\) in
\eqref{eq:h_soft}, we obtain
\[
\frac{e^{-ig_n(\zeta,\eta)}}{n\Delta Q(z_0)} K_n^{\sharp}(z_n,w)=
 \frac{1}{\sqrt{2\pi}}\int\limits_{\frac{m_n}{\nu(\Lambda)\sqrt{n}}}^{\frac{M_n}{\nu(\Lambda)\sqrt{n}}} 
\Phi(-t)^{-1}e^{-(\re \zeta)^2 -(\re \eta)^2} e^{-t(\zeta +\bar{\eta})} e^{-\frac{t^2}{2}} dt
+ \bigO\Big(\frac{\log^3 n}{\sqrt{n}}\Big).
\]   
Rewriting the integral gives
$$
\frac{c_n(\zeta,\eta) }{n\Delta Q(z_0)} K_n(z_n,w) = f_{n,\zeta}(\eta) 
+ \bigO\Big(\frac{\log^3 n}{\sqrt{n}}\Big), 
$$
where the function $f_{n,\zeta}$ be given by 
$$
f_{n,\zeta} (\eta) = G(\zeta,\eta) \frac{1}{\sqrt{2\pi}}
\int\limits_{\frac{m_n}{\nu(\Lambda)\sqrt{n}}}^{\frac{M_n}{\nu(\Lambda)\sqrt{n}}} 
\Phi(-t)^{-1} e^{-(\zeta + \bar{\eta} +t)^2/2} dt.
$$
The locally uniform convergence now follows since $\frac{M_n}{\sqrt{n}} \to \infty$ 
and $\frac{m_n}{\sqrt{n}} \to 0$ as $n\to \infty$.

We now turn to the norm estimate. Just as in the hard-edge case, the contribution 
from $w=w(\eta)\in \mathbb{D}(z_0,\delta_n)$ such that $\eta\in \mathbb{H}$
is negligible. This follows since we have that $|K_n(z,w)|^2 \leq Cn^2$, so we get
$$
\frac{1}{n\Delta Q(z_0)}\int\limits_{\mathbb{D}(z_0,\delta_n)} |K_n^{\sharp}(z_n,w)|^2 
1_S(w) 1_{\mathbb{H}}(\eta) dA(w) \lesssim \frac{1}{n} n^2 \delta_n^3  =o(1),
$$
as $n\to\infty$. The remaining part of the integral can be estimated as in the soft-edge case, and we get 
\begin{align}
\frac{1}{n\Delta Q(z_0)}\int\limits_{\mathbb{D}(z_0,\delta_n)} |K_n^{\sharp}(z_n,w)|^2 
1_S(w) 1_{\mathbb{L}}(\eta) dA(w) &\leq  \int\limits_{\mathbb{L}} |f_{n,\zeta}(\eta)|^2 dA(\eta) + o(1)\\
&\leq  \int\limits_{\mathbb{L}} |H_{\mathbb{L}}(\zeta,\eta)|^2 dA(\eta) + o(1)\\
&= H_{\mathbb{L}}(\zeta,\zeta) + o(1),
\end{align}
as $n\to \infty$.  This completes the proof. 
\end{proof}

It remains to show that the approximate kernel is small outside the $\delta_n$-neighborhood 
of the boundary point $z_0$. We denote by $\chi$ a smooth cut-off function equal to $1$ on 
a neighborhood of the unbounded component of \(\C\setminus S\). 
We require that this neighborhood is sufficiently small that the kernel 
\(K_n^\sharp(z,w)\) is well-defined wherever \(\chi\) does not vanish.

The following lemma applies to both the soft and soft/hard-edge kernels. The proof is inspired by
\cite{AmeurCronvall}.

\begin{lem}\label{lem:off_diag_softhard} For fixed $\zeta \in \C$ we have
$$
\frac{1}{n\Delta Q(z_0)}\int\limits_{\C \setminus \D(z_0,\delta_n)} 
|K_n^{\sharp}(z_n,w) \chi(w)|^2 dA(w) =
\mathrm{o}(1),
$$ 
as $n \to \infty$.
\end{lem}

\begin{proof} Let $z = z_n$. Note that we can write
$$
\int\limits_{\C \setminus \D(z_0,\delta_n)} |K_n^{\sharp}(z,w) \chi(w)|^2 dA(w) 
=  \int\limits_{\C \setminus \D(z_0,\delta_n)} \chi(w) |S_1(z,w)|^2 
e^{n(V-Q)(z)}e^{n(V-Q)(w)} dA(w),
$$
where $S_1$ is given by
$$
S_1(z,w)= \sum\limits_{j=m_n}^{M_n} h_{j,n}e^{-j(\mathcal{U}_\lambda(z) 
+ \overline{\mathcal{U}_\lambda(w)})}.
$$
In order to estimate the above integral, we define 
$$
a_j = h_{j,n} \quad \text{and}\quad b_j=b_j(z,w)= e^{-j(\mathcal{U}_\lambda(z_n) 
+ \overline{\mathcal{U}_\lambda(w)})},
$$
where $h_{j,n}$ is defined in \eqref{eq:h_soft}. We now have
$$
S_1= \sum\limits_{j=m_n}^{M_n} a_j b_j,
$$
and summation by parts gives 
\begin{equation}\label{sum_by_parts_softhard}
S_1= a_{M_n} B_{M_n} - a_{m_n} B_{m_n-1} -\sum\limits_{j=m_n}^{M_n-1} (a_{j+1}-a_j) B_j,
\end{equation}
where $B_j = \sum\limits_{k=0}^j b_k$. The sum $B_j$ can be written as
$$
B_j = \frac{1-\Theta^{j+1}}{1-\Theta},
$$
where $\Theta = \Theta(z,w) =e^{-\mathcal{U}_\lambda(z)-\overline{\mathcal{U}_\lambda(w)}}$. 
To estimate the integral in the statement of the lemma we will estimate the integral 
of each of the terms in \eqref{sum_by_parts_softhard}. We will start by considering 
integrals of the form
  $$
I_j = \int\limits_{\mathcal{N}_n}  | B_j|^2  e^{n(V-Q)(z)}e^{n(V-Q)(w)} dA(w),
$$
where $\mathcal{N}_n$ is the intersection of $\C\setminus \mathbb{D}(z_0,\delta_n)$ 
with a small fixed neighborhood \(\mathcal{N}\) of $\Gamma$. 
To estimate $I_j$ we integrate over the level curves $\Gamma_t$ where 
$(V-Q)(w)=-t^2/2$ with $t\in(-\epsilon,\epsilon)$. 
For \(w\in \Gamma_t\) with \(t\) small, there exists 
a unique closest point \(w_0\in \partial S\) to \(w\), 
and an application of Taylor's formula shows that
\[
w=w_0+\frac{e^{i\theta}t}{2\sqrt{\Delta Q(z_0)}}+\bigO(t^2),\quad w\in \Gamma_t
\]
as \(t\to 0\).
By Lemma~\ref{lem:greeb-pot-soft} we thus find that
\[
\re \mathcal{U}_\lambda(w)=-\frac{t}{2\nu(\Lambda)}+\bigO(t^2),\quad w\in\Gamma_t
\]
as \(t\to 0\). From the local asymptotics of $\mathcal{U}_\lambda$ 
we have that $|e^{-j\mathcal{U}_\lambda(z_n)}| \leq C_\zeta$,
and hence we obtain the bound
$$
|\Theta|^j \leq C_\zeta e^{-j\frac{t}{2\nu(\Lambda)} + Cjt^2}, \quad w\in \Gamma_t.
$$ 

In addition, we have the elementary bound
$$
|\Theta-1|^2 \gtrsim |z-w|^{2},
$$
and combining this with the above we find that
\begin{align*}
I_j \lesssim \int\limits_{\Gamma_t} &1_{\{|z-w|>\delta_n\}} (w) |z-w|^{-2} ds_t(w) 
\int\limits_{-\epsilon}^{\epsilon} e^{-\frac{tj}{\nu(\Lambda)} + 2Ct^2 j} e^{-nt^2/2} dt,
\end{align*}
where \(ds_t\) denotes the arc-length measure on \(\Gamma_t\).
The first integral is of order $\bigO(\frac{1}{\delta_n})$. For the second integral we have
\begin{align*}
\int\limits_{-\epsilon}^{\epsilon} e^{-\frac{tj}{\nu(\Lambda)} + 2Ct^2 j} e^{-nt^2/2} dt 
&= \frac{1}{\sqrt{n}} \int\limits_{-\epsilon \sqrt{n}}^{\epsilon \sqrt{n}} e^{2Cjt^2/n} 
e^{-\frac{tj}{\nu(\Lambda)\sqrt{n}}} e^{-t^2/2} dt \\
&\lesssim  \frac{1}{\sqrt{n}} \int\limits_{-\epsilon \sqrt{n}}^{\epsilon \sqrt{n}}   
e^{\frac{j^2}{2\nu(\Lambda)^2 n}}  e^{ -\frac12(\frac{j}{\nu(\Lambda)\sqrt{n}} + t)^2} dt\\
&\lesssim \frac{1}{\sqrt{n}} e^{\frac{j^2}{2\nu(\Lambda)^2 n}}.
\end{align*}
Combining the two estimates we obtain $I_j \lesssim \frac{1}{\delta_n \sqrt{n}} 
e^{\frac{j^2}{2\nu(\Lambda)^2 n}}$. Using that 
$a_j^2 \lesssim n e^{-\frac{j^2}{\nu(\Lambda)^2n}}$ we find that
$$
\int\limits_{\mathcal{N}_n}  a_j^2 |B_j|^2 e^{n(V-Q)(z)}e^{n(V-Q)(w)} dA(w) 
\lesssim  a_j^2 I_j  \lesssim \frac{\sqrt{n}}{\delta_n} e^{-\frac{j^2}{2\nu(\Lambda)^2n}} = o(n),
$$ 
which shows that the integral of the absolute value squared of the first two terms 
in \eqref{sum_by_parts_softhard} are of the right order. It remains to estimate the 
integral of the absolute value squared of the sum in \eqref{sum_by_parts_softhard}. 
To begin with, from Cauchy-Schwarz we have
\begin{align}
\left\vert\sum\limits_{j=m_n}^{M_n-1} (a_{j+1}-a_j) B_j \right\vert^2 &
\lesssim M_n \sum\limits_{j=m_n}^{M_n-1} (a_{j+1}-a_j)^2 |B_j|^2 .
\end{align}
Using that $|a_{j+1}-a_j| \lesssim a_j (\frac{j}{n} + \frac{1}{\sqrt{n}}),$ we get 
\begin{align}
M_n \sum\limits_{j=m_n}^{M_n-1} (a_{j+1}-a_j)^2 I_j &\lesssim M_n  
\sum\limits_{j=m_n}^{M_n-1} \Big(\frac{j}{n} + \frac{1}{\sqrt{n}}\Big)^2 a_j^2 I_j \\
& \lesssim M_n \int\limits_0^\infty (1 + s)^2  
e^{-\frac{s^2}{2\nu(\Lambda)}}  ds  \cdot  \frac{1}{\delta_n}  \\
&\lesssim \frac{M_n}{\delta_n}  = o(n),
\end{align}
where the last step uses the definition \eqref{eq:def-mn-Mn-soft} of \(M_n\).

Outside the neighborhood \(\mathcal{N}\), the contribution to the 
norm of \(1_{\C\setminus\D(0,\delta_n)}K_n^\sharp(z_n,\cdot)\) will be small.
This is somewhat non-trivial to show, as the approximate kernel may have some 
\(L^2\)-mass distributed around distant boundary components of \(S\), i.e., 
around \(\partial S\setminus\Gamma\). However, this mass will be of a lower 
order than \(\bigO(n)\). To see this, we note that
$$
| e^{-j\mathcal{U}_{\lambda_j}(w)} |  = e^{-jU_{\mathcal{D}}^{\lambda_j}(w)}.
$$
For $w\in\mathcal{D}$ outside some fixed neighborhood of $\Gamma$ and $\zeta\in\Lambda$, 
we have the uniform lower bound $g_{\mathcal{D}}(w,\zeta)\ge c>0$. It follows that
\[
| e^{-j\mathcal{U}_{\lambda_j}(w)} | =  \bigO(e^{-cj}).
\]
This gives that \(|S_1|=\bigO(\sqrt{n})\), and hence 
\begin{align}
\int_{\C\setminus\mathcal{N}}|\chi(w)\, K_n^\sharp(z_n,w)|^2 dA(w)
&=\int_{\C\setminus\mathcal{N}}|\chi(w)\, S_1(z_n,w)|^2 e^{n(V-Q)(z_n)+n(V-Q)(w)}dA(w)\\ 
&\lesssim n\int_{\C\setminus\mathcal{N}}\chi^2(w) e^{n(V-Q)(z_n)+n(V-Q)(w)}dA(w)=\bigO(\sqrt{n}).
\end{align}
This completes the proof.
\end{proof}

\subsection{Correction to polynomial}
We want to show that the approximate kernel $K_n^{\sharp}(\cdot,w)$ can 
be approximated in a suitable way by weighted polynomials.
We will consider points $w = w_n$, and try to find a function $u_w$ such that 
\begin{equation}\label{eq:correction_soft}
\chi(z)K_n^{\sharp}(z,w) e^{\frac{n}{2}Q(z)} - u_w(z)
\end{equation}
is a polynomial of degree at most $n-1$. The correction $u_we^{-\frac{n}{2}Q}$ 
should be small both in norm and pointwise near $z_0$. 

If we can show that
\begin{equation}\label{eq:norm_soft_der}
\int\limits_{\C} \left|\bar\partial_z \Big( K_n^{\sharp}(z,w)e^{\frac{n}{2}Q(z)} 
\chi(z) \Big) \right|^2 e^{-nQ(z)} dA(z) = \bigO(e^{-cn}),
\end{equation}
for some $c>0$ a Hörmander estimate will give the existence of $u_w$ such that
$$
\|u_w\|_{\mathcal{W}_n}^2 = \bigO(e^{-c'n}), 
$$
for some $c'>0$ and $u_w(z) = \bigO(z^{n-1})$ as $|z|\to \infty$. Since 
$K_n^{\sharp}(z,w)e^{\frac{n}{2}Q(z)} = \bigO(z^{n-1})$ as $|z|\to\infty$ 
it will follow that \eqref{eq:correction_soft} defines a polynomial of degree at most $n-1$.

The derivative in \eqref{eq:norm_soft_der} is non-zero only where the 
complex gradient $\bar\partial_z \chi(z) \neq 0$. It follows that it is enough to estimate the integral
$$
\int\limits_{N} |K_n^{\sharp}(z,w)|^2 dA(z),
$$
where $N = \{z : \bar\partial_z \chi(z) \neq 0\}$. From the definition of $K_n^{\sharp}(z,w)$ 
and Cauchy-Schwarz's inequality (for sums) 
we get 
$|K_n^{\sharp}(z,w) |^2 \leq K_n^{\sharp}(z,z) K_n^{\sharp}(w,w)$ and 
from local estimates from before $K_n^{\sharp}(w_n,w_n) = \bigO(n^2)$.

Let us first consider the part of $N$ in the component of $S$ to 
which $z_0$ belongs. We use the estimate 
$$
K_n^{\sharp}(z,z) \lesssim M_n \max_{m_n\leq j \leq M_n}
e^{-2jU_\mathcal{D}^{\lambda}(z)} e^{n(V-Q)(z)}
$$
First note that 
$$
e^{n(V-Q)(z)} = \bigO(e^{-cn}),
$$
uniformly on $N$. Since $j <  M_n =\bigO(\sqrt{n}\log n)$ we get that
$$
K_n^{\sharp}(z,z) = \bigO(e^{-c'n})
$$
for some other positive constant \(c'>0\).
On the other components of \(S\) the estimates are even simpler, since we have
$$
e^{-2jU_\mathcal{D}^{\lambda_j}(z)}  = \bigO(e^{-cj }),  
$$
uniformly for $z\in\mathcal{D}$ such that $\text{dist}(z,\partial \mathcal{D})>d>0$. 
This finalizes the proof of the estimate \eqref{eq:norm_soft_der}, 
and the existence of suitable corrections \(u_w\) follows.

\subsection{Proof of Theorem \ref{thm:main-soft-edge} and 
Theorem \ref{thm:main-soft-hard-edge}}

The proofs will now be almost identical to the proof of 
Theorem~\ref{thm:main-hard-edge}. We will outline the steps 
for Theorem \ref{thm:main-soft-edge}. The corresponding steps for the 
soft/hard edge are written in parenthesis. In the soft/hard-edge case, 
$H$ should be replaced by $H_{\mathbb{L}}$.

Let $K(\zeta,\eta)$ be a limit point of the rescaled kernel along 
the subsequence $\{n_k\}_k$. From Corollary \ref{cor: upper_bound_soft} 
(Corollary \ref{cor:upper_bound_soft_hard}) 
we get that $K(\zeta,\zeta)\leq H(\zeta,\zeta)$.

Since $K_n$ is the reproducing kernel of the space $\mathcal{W}_n$, we have that
$$
K_n(z,z) = \sup\limits_{f\in\mathcal{W}_n\setminus \{0\}} 
\frac{|f(z)|^2 }{\phantom{a}\| f\|_{\mathcal{W}_n}^2}.
$$ 
We have proven that $K_n^{\sharp}\chi-u_{z_n}$ is a polynomial of degree at most $n-1$. 
Since $u_{z_n}e^{-\frac{n}{2}Q}$ is small both pointwise near $z_n$ and in norm we get
\begin{equation}\label{eq_pol_rep_soft}
K_n(z_n,z_n) \geq (1+o(1))\frac{K_n^{\sharp}(z_n,z_n)^2}{\phantom{i}
\|K_n^{\sharp}(\cdot,z_n)\chi(\cdot)\|_{\mathcal{W}_n}^2}.
\end{equation}
From Theorem \ref{thm:limiting_soft} (Theorem \ref{thm:limiting_softhard}) 
and Lemma \ref{lem:diagonal_approx}  (Lemma \ref{lem:diagonal_approx_softhard}) we have that
$$
\lim\limits_{k\to\infty} \frac{1}{n_k\Delta Q(z_0)} K_{n_k}(z_{n_k},z_{n_k}) = K(\zeta,\zeta)
$$
and
$$
\lim\limits_{k\to\infty} \frac{1}{n_k\Delta Q(z_0)} 
K_{n_k}^{\sharp}(z_{n_k},z_{n_k}) = H(\zeta,\zeta).
$$
Lemma \ref{lem:diagonal_approx} (Lemma \ref{lem:diagonal_approx_softhard}) 
and Lemma \ref{lem:off_diag_softhard} gives that
$$
\lim\limits_{k\to \infty} \frac{1}{n_k\Delta Q(z_0)}
\|K_{n_k}^{\sharp}(\cdot,z_{n_k})\chi(\cdot)\|_{\mathcal{W}_n}^2 \leq H(\zeta,\zeta).
$$
Using these properties together with \eqref{eq_pol_rep_soft} gives the lower bound
$$
K(\zeta,\zeta)\geq H(\zeta,\zeta).
$$
We get that any limiting kernel $K(\zeta,\eta)$ is equal to $H(\zeta,\eta)$, 
which completes the proof.
\hfill \(\square\)

\subsection*{Acknowledgements}
The research of Joakim Cronvall was supported by Grant No.\ 2024.0309 
from the Knut and Alice Wallenberg foundation, and Aron Wennman was 
supported by Odysseus Grant G0DDD23N from Research Foundation Flanders (FWO).

\bibliography{bibtex-local.bib}

\bigskip
\bigskip

\noindent
\begin{minipage}[t]{0.55\textwidth}
\noindent \sc Joakim Cronvall\newline
Department of Mathematics \newline
KU Leuven, Belgium
\newline {\tt joakim.cronvall@kuleuven.be}
\end{minipage}
\hfill
\begin{minipage}[t]{0.45\textwidth}
\noindent \sc Aron Wennman\newline
Department of Mathematics \newline
KU Leuven, Belgium
\newline {\tt aron.wennman@kuleuven.be}
\end{minipage}
\end{document}